\documentclass[12pt]{amsart}

\usepackage{amsfonts}
\usepackage{amsmath}
\usepackage{amssymb}
\usepackage{amsthm} 
\usepackage{mathrsfs}
\usepackage{xcolor}
\usepackage{appendix}
\usepackage{bm}
\usepackage{tikz}
\DeclareOldFontCommand{\rm}{\normalfont\rmfamily}{\mathrm}

\newcommand{\thebottomline}{\renewcommand{\thefootnote}{}\renewcommand{\footnoterule}{}\phantom{M}{\hfill{}\footnotetext{\noindent\textit{\tiny\romannumeral\day.\romannumeral\month.\romannumeral\year}}}}

\newcommand\utimes{\mathbin{\ooalign{$\cup$\cr%
   \hfil\raise0.42ex\hbox{$\scriptscriptstyle\times$}\hfil\cr}}}
\newcommand\bigutimes{\mathop{\ooalign{$\bigcup$\cr%
   \hfil\raise0.36ex\hbox{$\scriptscriptstyle\boldsymbol{\times}$}\hfil\cr}}}
\usepackage[T1]{fontenc}
\usepackage{mathabx,graphicx}

\allowdisplaybreaks
\newcommand{\ab}{\allowbreak}
\theoremstyle{definition}
\newtheorem{thm}{Theorem}[section] 
\newtheorem{cor}[thm]{Corollary}
\newtheorem{lem}[thm]{Lemma}
\newtheorem{prop}[thm]{Proposition}
\theoremstyle{definition}
\newtheorem{defn}[thm]{Definition}
\newtheorem{rem}[thm]{Remark}
\newtheorem{nota}[thm]{Notation}

\newcommand{\cA}{\mathcal{A}}
\newcommand{\cB}{\mathcal{B}}
\newcommand{\cD}{\mathcal{D}}
\newcommand{\cP}{\mathcal{P}}

\newcommand{\ds}{\displaystyle}

\newcommand{\bC}{\mathbb{C}}
\newcommand{\bR}{\mathbb{R}}

\newcommand{\alg}{\mathrm{alg}}
\newcommand{\rE}{\mathrm{E}}
\newcommand{\Tr}{\mathrm{Tr}}
\newcommand{\tr}{\mathrm{tr}}
\newcommand{\comp}{\circ}
\newcommand{\NC}{\mathit{NC}}
\newcommand{\CI}{\mathit{CI}}

\newcommand{\sss}{\scriptscriptstyle}
\newcommand{\sN}{{\scriptscriptstyle N}}
\let\phi=\varphi
\let\epsilon=\varepsilon
\newcommand{\mybullet}{$\vcenter{\hbox{$\scriptscriptstyle\bullet$}}$}

\newcommand{\diag}{\textrm{diag}}

\newcommand{\jp}{j^\perp}

\newcommand\blfootnote[1]{%
     \begingroup
     \renewcommand\thefootnote{}\footnote{#1}%
     \addtocounter{footnote}{-1}%
      \endgroup
}

     \makeatletter
    \def\footnoterule{\kern-3\p@
      \hrule \@width 2in \kern 2.6\p@} 
    \makeatother

\numberwithin{equation}{section} 

\author[Mingo]{James A. Mingo}
\address[J.~A.~Mingo]{Department of Mathematics and Statistics, 
                      Queen's University, 
                      Jeffery Hall,  \ab
                      Kingston, \ab Ontario,  
                      Canada K7L 3N6 
}
\email{mingo@mast.queensu.ca}

\author[Tseng]{Pei-Lun Tseng}
\address[P.-L. Tseng]{Department of Mathematics \\ 
New York University Abu Dhabi \\
Saadiyat Marina District, Abu Dhabi, United Arab Emirates
}
\email{pt2270@nyu.edu}

\title[Infinitesimal Operators]{Infinitesimal Operators and the Distribution of Anticommutators and Commutators}

\begin{document}

\begin{abstract}
In an infinitesimal probability space we consider operators which are infinitesimally free and one of which is \textit{infinitesimal}, in that all its moments vanish. Many previously analysed random matrix models are captured by this framework. We show that there is a simple way of finding non-commutative distributions involving infinitesimal operators and apply this to the commutator and anticommutator. We show the joint infinitesimal distribution of an operator and an infinitesimal idempotent gives us the Boolean cumulants of the given operator. We also show that Boolean cumulants can be expressed as infinitesimal moments thus giving matrix models which exhibit asymptotic Boolean independence and monotone independence.
\end{abstract}

\maketitle

\blfootnote{Research was supported by a Discovery Grant from the Natural Sciences and Engineering Research Council of Canada and the Fields Institute for Research in the Mathematical Sciences.}


\section{Introduction}\label{sec:introduction}

The largest eigenvalues of a random covariance matrix has been much studied in recent years, motivated in many cases by problems in principle component analysis. An important question in principal component analysis is to decide how many eigenvalues can be considered significant. To analyse this problem a spike or signal-plus-noise model is often used where we have a source of noise represented by a random matrix $X_1$ and a signal represented by a non-random matrix $X_2$. Moreover it is usually assumed that the rank of the signal is much much smaller than the rank of the noise, which increases with the sample size.   Even though the rank of $X_2$ is small, it still affects the largest eigenvalue of $X_1 + X_2$. In this paper we idealize this by saying that the signal, $X_2$, is \textit{infinitesimal}.

The tools of free probability give explicit rules (see \cite{V, MS, NS}) for computing the distributions of the sum or product of freely independent random variables. The method of linearization extends these rules to the case of polynomials (see \cite{BS} and \cite[Ch.~10]{MS})  or rational functions (see \cite{msy}) of freely independent random variables, but explicit solutions are often infeasible.  

An example where explicit calculations are feasible is the case of a commutator: $i(x_1 x_2 - x_2 x_1)$ and an anti-commutator: $x_1 x_2 + x_2 x_1$. In \cite[Lect.~19]{NS} the distribution of the commutator of freely independent operators is given; the case of the anticommutator is more involved (see \cite[Rem.~15.21]{NS}); however a graph theoretic description was given by Perales in \cite{Per21}. A related case is the recent work of Campbell \cite{Camp22} on the finite free distribution of a commutator an anticommutator. Also relevant are the papers of Ejsmont, Lehner \cite{EL,EL2}, and H{\k{e}}{\'c}ka \cite{EH} on the tangent law. 

In \cite{BHSakuma18} Collins, Hasebe, and Sakuma investigated the spectrum of the commutator $i(X_1X_2 - X_2  X_1)$ and the anticommutator $X_1X_2 + X_2  X_1$ when the rank $X_1$ tended to infinity and the rank of $X_2$ was fixed. The regime of  \cite{BHSakuma18} can be modelled by the infinitesimal free independence of Belinschi and Shlyakhtenko \cite{BS} and its application to study asymptotic independence given by Shlyakhtenko \cite{shly}. Expanding on this Cébron and Dahlqvist, and Gabriel \cite{GAF22}, used infinitesimal freeness to give a rigorous prove the BBP phase transition phenomena of outliers found by Baik, Ben Arous, and Péché \cite{bbp}. Finally in \cite{TS22} it was shown that infinitesimal free independence was just a special case of freeness with amalgamation over a commutative algebra.  This was the inspiration for our work. As we shall see in this paper, many of the algebraic and combinatorial tools  of \cite{NS} transfer directly to the infinitesimal case. Moreover this transfer also extends to the  Boolean independence of Speicher and Wourodi \cite{SW97}, as was shown by Perales and Tseng in \cite{PT21}. An important point to notice here is that our methods give the full distribution of the commutator and anticommutator, not just the support of the distribution as was done in \cite{BHSakuma18}.

From the point of view of free probability, the key property of $X_2$ used in \cite{BHSakuma18} was that its limit distribution was $\delta_0$, the Dirac mass at $0$. An operator with this property we call \textit{infinitesimal}. 

There has been much recent work connecting free and Boolean cumulants (see for example the papers of Février, Mitja, Nica, and Szpojankowski \cite{fmns} and \cite{fmns1}). We show that if one further assumes  that  the infinitesimal operator is an idempotent. Then there are new and simple formulas relating infinitesimal moments to Boolean cumulants (see Prop.~\ref{prop:boolean cumulants}). This enables us to  produce matrix models exhibiting asymptotic Boolean independence and asymptotic monotone independence, which until now has been difficult to do, see \cite{l} and \cite{male}. Also in Remark \ref{rem:spectral shift} we show how infinitesimal operators give an abstraction for the Markov-Krein transform of Kerov \cite{kerov}.

In \S\ref{sec:preliminaries} we collect the basic notions from \cite{NS}, \cite{FN}, and \cite{MS} that we shall need for our results. In \S\ref{sec:some polynomials} we consider the distribution of the commutator and the anticommutator when $x_1$ and $x_2$ are infinitesimally free and $x_2$ is infinitesimal. In \S\ref{sec:R-diagonal} we consider infinitesimal $R$-diagonal operators. In \S\ref{sec:idempotents} we show how to write a Boolean cumulant as an infinitesimal moment. In \S\ref{sec:boolean-independent} we revisit the  problem in \S\ref{sec:some polynomials} with infinitesimal free independence replaced by Boolean independence. 


\section{Preliminaries}\label{sec:preliminaries}

\subsection{Free Independence}

A \textit{non-commutative probability space} $(\mathcal{A},\varphi)$ consists of a unital algebra over $\mathbb{C}$ and a linear functional $\varphi:\mathcal{A}\to\mathbb{C}$ with $\varphi(1)=1$. We let $\mathbb{C}\langle X_1,\dots,X_k\rangle$ be the algebra on polynomials in the non-commuting variables $X_1, \dots , X_n$.

For $a_1,\dots, a_k$ be elements in $\mathcal{A}$, the mapping $\mathbb{C}\langle X_1,\dots,X_k\rangle \ni p \mapsto \phi(p(a_1, \dots, a_k)) \in \bC$ is the \textit{joint distribution} of $\{a_1,\dots, a_k\}$. We denote this linear functional by $\mu_{a_1, \dots, a_k}$.

For a given non-commutative probability space $(\mathcal{A},\varphi)$, unital subalgebras $(\mathcal{A}_i)_{i\in I}$ of $\mathcal{A}$ are \textit{freely independent} (or \textit{free} for short) if for all $n\in \mathbb{N}$, $a_1,\dots, a_n\in \mathcal{A}$ are such that $a_j\in\mathcal{A}_{i_j}$ where $i_1,\dots, i_n\in I$, $i_1\neq i_2\neq \cdots \neq i_n$ and $\varphi(a_j)=0$ for all $j$, then
$$
\varphi(a_1\cdots a_n)=0.
$$
A set $\{a_i\in\mathcal{A}\mid i\in I\}$ is free if the unital subalgebras $\alg(1,a_i)$ generated by $a_i\ (i\in I)$ form a free family. 

For $n\geq 1$, a partition of $[n]$ is a set $\pi=\{V_1,\dots, V_r\}$ of pairwise disjoint non-empty subsets of $[n]:=\{1,2,\dots, n-1,n\}$ such that 
$V_1\cup \cdots\cup V_r=[n]$. The set of all partitions of $[n]$ is denoted by $\cP(n)$, and $NC(n)$ denotes all non-crossing partitions of $[n]$ in the sense that we cannot find distinct blocks $V_r$ and $V_s$ with $a, c\in V_r$ and $b, d\in V_s$ such that $a<b<c<d$. 

\begin{nota} Let $(\cA, \phi)$ be a non-commutative probability space and $a \in \cA$. Let $m_n(a) = \phi(a^n)$ be the \textit{moments} of $a$. 
Suppose $a_1,\dots, a_n$ are elements in $\mathcal{A}$ and $V=\{i_1<\cdots<i_s\}$ is a block of some partition of $[n]$, we set $(a_1,\dots, a_n)|_V=(a_{i_1},\dots, a_{i_s}).$ Let $\{f_n\}_{n\geq 1}$ is a sequence of multilinear functionals $f_n:\mathcal{A}^n\to\mathbb{C}$, then for each partition $\pi$, we set
$$
f_{\pi}(a_1,\dots,a_n)=\prod_{V\in \pi}f_{|V|}((a_1,\dots,a_n)|_V).
$$
\end{nota}
The \textit{free cumulants} $\{\kappa_n:\mathcal{A}^n\to\mathbb{C}\}_{n\geq 1}$ are defined inductively in terms of moments via 
$$
\varphi(a_1\cdots a_n)=\sum_{\pi\in NC(n)}\kappa_{\pi}(a_1,\dots,a_n).
$$
Note that the free cumulants can also be defined via the following formula 
$$
\kappa_n(a_1,\dots,a_n)=\sum_{\pi\in NC(n)}\mu(\pi,1_n)\varphi_{\pi}(a_1,\dots,a_n),
$$
where $\mu$ is the Möbius function of $NC(n)$ (see \cite[Lect.~11]{NS}).

Free cumulants are crucial to characterizing freeness. More precisely, unital subalgebras $(\mathcal{A}_i)_{i\in I}$ are freely independent if and only if for all $n\geq 2$, and for all $a_j\in \mathcal{A}_{i_j}$ for all $j=1,\dots,n$ with $i_1,\dots, i_n\in I$ we have $\kappa_n(a_1,\dots, a_n)=0$ whenever there exist $1\leq l, k \leq n$ with $i_l\neq i_k$ (see \cite[Thm.~11.16]{NS}). This characterization of freeness is called the \textit{vanishing of mixed cumulants}.

For a given $a\in \mathcal{A}$, the \textit{Cauchy transform} of $a$ is defined by
$$
G_a(z)=\sum_{n=0}^{\infty}\frac{m_n(a)}{z^{n+1}},
$$
and the \textit{$R$-transform} of $a$ is defined as 
$$
R_a(z)=\sum_{n=0}^{\infty}\kappa_{n+1}(a)z^{n} \text{ where }\kappa_n(a)=\kappa_n(a,\dots, a). 
$$
Note that following the vanishing of mixed cumulants property, we have if $a,b\in \cA$ are free, then
$R_{a+b}=R_a+R_b$. We shall always work with formal power series, however if our distributions arise from Borel measures on $\bR$ with compact support, then all of our series converge to an analytic function on an appropriate domain.

\subsection{$\bm{R}$-diagonal Operators}\label{Pre: R-diagonal}

A non-commutative probability space $(\cA, \phi)$ is called a \emph{$*$-probability space} if $\cA$ is a $*$-algebra and $\phi$ is positive in the sense that $\phi(a^*a)\geq 0$ for all $a\in\cA$.  
For $\epsilon \in \{-1, 1\}$ and $a \in \cA$ we let $a^{(\epsilon)} = a$ if $\epsilon = 1$ and $a^{(\epsilon)} = a^*$ if $\epsilon = -1$. Recall from \cite[Lect.~15]{NS} that $a \in\cA$ is $R$-\textit{diagonal} if for all $\epsilon_1, \dots, \epsilon_n \in \{-1, 1\}$ we have 
\[
\kappa_n(a^{(\epsilon_1)},\dots,a^{(\epsilon_n)})=0 
\]
unless $n$ is even and $\epsilon_i + \epsilon_{i+1} = 0$ for $1 \leq i < n$.  In addition, we recall that $a \in \cA$ is \textit{even} if $a = a^*$ and $\varphi(a^{2n+1})=0$ for all $n \geq 0$. 

Note that the $R$-diagonal element is uniquely determined by its determining sequences (see \cite[Cor.~15.7]{NS}). To be precise, if $a$ is an $R$-diagonal element in a $*$-probability space $(\mathcal{A},\varphi)$, then the $*$-distribution of $a$ can be uniquely determined by $(\alpha_n)_n$ and $(\beta_n)_n$ where
$$
\alpha_n=\kappa_{2n}(a,a^*,a,a^*,\dots, a,  a^*) \text{ and }\beta_n=\kappa_{2n}(a^*,a,a^*,a\cdots, a^*,a).
$$
Moreover, the distribution of $aa^*$ and $a^*a$ can be described by its determining sequences $\{\alpha_n\}_n$ and $\{\beta_n\}_n$ in the following relation

\begin{eqnarray*}
    \kappa_n(aa^*,\dots,aa^*)&=&
    \sum_{\substack{\pi\in NC(n)\\
                  \pi=\{V_1,\dots,V_r\}}} \alpha_{|V_1|}\beta_{|V_2|}\cdots \beta_{|V_r|} \\
\kappa_n(a^*a,\dots,a^*a)&=&\sum_{\substack{\pi\in NC(n)\\
                  \pi=\{V_1,\dots,V_r\}}}                 \beta_{|V_1|}\alpha_{|V_2|}\cdots \alpha_{|V_r|}.
\end{eqnarray*}
where $V_1$ is the block of $\pi$ which contains the first element $1$. 

In addition, if $a$ and $b$ are $*$-free elements in a $*$-probability space such that $a$ is $R$-diagonal, then $ab$ is also $R$-diagonal (see \cite[Prop.~15.8]{NS}).

\subsection{Infinitesimal Free Independence}

\begin{defn}
Let $(\cA, \phi)$ be a non-commutative probability space. Let $\phi': \cA \rightarrow \bC$ be a linear functional such that $\phi'(1) = 0$. Then we call the triple $(\cA, \phi, \phi')$ an \textit{infinitesimal probability space}. 
\end{defn}
For $a_1,\dots,a_k\in \cA$, the pair of linear functionals $(\mu,\mu')$ which both map $\mathbb{C}\langle X_1,\dots,X_k\rangle$ into $\mathbb{C}$ are the \textit{joint infinitesimal distribution} of $\{a_1,\dots,a_k\}$ if $\mu=\mu_{a_1,\dots,a_k}$ and
$$
\mu'(p)=\varphi'(p(a_1,\dots,a_k)).
$$
Similarly, we denote the linear functional $\mu'$ by $\mu'_{a_1,\dots,a_k}.$

\begin{defn}
Let $(\cA, \phi, \phi')$ be an infinitesimal probability space, unital subalgebras $(\cA_i)_{i\in I}$ of $\cA$ are \textit{infinitesimally freely independent} (or \textit{infinitesimally free} for short) with respect to $(\varphi,\varphi')$ if whenever $a_1,\dots, a_n\in\cA$ are such that   
$a_k\in\cA_{i_k},i_1,\dots,i_n\in I,i_1\neq \cdots \neq i_n$ and $\varphi(a_j)=0$ for all $j$, then we have 
\begin{eqnarray*}
    \varphi(a_1\cdots a_n)&=&0 \\
    \varphi'(a_1\cdots a_n)&=&\sum_{k=1}^n\varphi'(a_k)\varphi(a_1\cdots a_{k-1}a_{k+1}\cdots a_n). 
\end{eqnarray*}
\end{defn}

\begin{nota}
Let $(\cA, \phi, \phi')$ be an infinitesimal probability space and $a \in \cA$. Let $m_n'(a)=\varphi'(a^n)$ be the infinitesimal moments of $a$. 
Let $\{f_n:\cA^n\to\mathbb{C}\}_{n\geq 1}$ and $\{f_n':\cA^n\to\mathbb{C}\}_{n\geq 1}$ be two sequences of multi-linear functionals. Given a partition $\pi\in NC(n)$, then for a fixed $V\in \pi$, we define $\partial f_{\pi,V}$ be the map that is equal to $f_{\pi}$ except for the block $V$, where we replace $f_{|V|}$ by $f'_{|V|}$. Then we define $\partial f_{\pi}$ by
$$
\partial f_{\pi}(a_1,\dots,a_n)= \sum_{V\in \pi}\partial f_{\pi,V}(a_1,\dots,a_n).
$$
\end{nota}

\begin{defn}
Suppose $(\cA,\varphi,\varphi')$ is an infinitesimal probability space, the \textit{infinitesimal free cumulants} $\{\kappa_n':\cA^n\to\mathbb{C}\}$ are defined inductively in terms of moments and infinitesimal moments via  
$$
\varphi'(a_1,\dots,a_n)=\sum_{\pi\in NC(n)}\partial \kappa_{\pi}(a_1,\dots,a_n). 
$$
\end{defn}

Note that the infinitesimal free cumulants can also be described via 
$$
\kappa_n'(a_1,\dots,a_n)=\sum_{\pi\in NC(n)}\mu(\pi,1_n)\partial \varphi_{\pi}(a_1,\dots,a_n)
$$
(see \cite{FN,M}).
The infinitesimal free cumulants can be used to characterize infinitesimal freeness. 
\begin{thm}\cite[Thm.~1.2]{FN}
Suppose that $(\cA,\varphi,\varphi')$ is an infinitesimal probability space, and $\cA_i$ is a unital subalgebra of $\cA$ for each $i\in I$. Then $(\cA_i)_{i\in I}$ are infinitesimally free if and only if for each $s\geq 2$ and $i_1,\dots, i_s\in I$ which are not all equal, and for $a_1\in \cA_{i_1},\dots,a_s\in \cA_{i_s}$, we have $\kappa_n(a_1,\dots,a_s)=\kappa_n'(a_1,\dots,a_s)=0.$
\end{thm}
For a given $a\in \cA$, we define the \textit{infinitesimal Cauchy transform }of $a$ by 
$$
g_a(z)=\sum_{n=1}^{\infty}\frac{m_n'(a)}{z^{n+1}},
$$
and the \textit{infinitesimal $r$-transform }of $a$ is defined by 
$$
r_a(z) = \sum_{n=0}^{\infty}\kappa'_{n+1}(a)z^{n} \text{ where }\kappa'_n(a)=\kappa'_n(a,\dots,a). 
$$
Following \cite{M}, the infinitesimal Cauchy and $R$-transforms are related by the equations 
\begin{equation}\label{eq:inf_cauchy_r_relation}
r_a(z)=-g_a\left(G^{\langle -1\rangle}_a(z)\right)(G^{\langle -1 \rangle}_a)'(z)  \text{ and } g_a(z)=-r_a(G_a(z))G_a'(z)
\end{equation}
where $G^{\langle -1\rangle}_a$ is the composition inverse of $G_a$.

\subsection{Applications to Random Matrices}
In \cite[\S4]{BHSakuma18} Collins, Hasebe and Sakuma proved asymptotic infinitesimal freeness for some random matrix ensembles, using the language of cyclic monotone independence. In particular they considered two ensembles of $N \times N$ matrices $\{A_{1,N}, \dots, A_{k,N}\}$ and $\{B_{1,N}, \dots, B_{l,N}\}$.  Let $\phi_\sN = \rE \comp \tr$ and $\phi'_\sN = \rE \comp \Tr$ where $\Tr$ is the un-normalized trace and $\tr = N^{-1} \Tr$ is the normalized trace. They assumed that there was a limit joint  distribution for the $B_{i,N}$'s with respect to $\phi_\sN$ and a limit joint distribution of the $A_{j,N}$'s with respect to $\phi'_\sN$. More precisely that for all $i_1, \dots, i_m \in [l]$ the limit $\lim_N \phi_\sN(B_{i_1,N} \cdots B_{i_m,N})$ exists, and for all $j_1, \dots, j_n \in [k]$ the limit $\lim_N \phi'_\sN(A_{j_1,N} \cdots A_{j_n,N})$ exists. 

Then in  \cite[Thm.~4.1]{BHSakuma18}, they showed that when $\{U_N\}_N$ is a Haar unitary independent from $\{A_{1,N}, \dots,\ab A_{k,N}, \ab B_{1,N}, \dots, B_{l,N}\}$ then the ensembles $\{U_N B_{1,N} U_N^*, \dots, \ab U_N B_{l,N} U_N^*\}$ and $\{A_{1,N}, \dots, \ab A_{k,N}\}$ are asymptotically infinitesimally free with respect to $(\phi, \phi')$. The special case when the $A_{i,N}$'s are stationary was already presented by Shlyakhtenko in \cite{shly}. In both these models the limit operators of $\{A_{1,N}, \dots, A_{k,N}\}$ are infinitesimal operators in the sense of Definition \ref{def:infinitesimal}, so the limit distributions of this model are covered by our theory. In addition, the results of \cite[Thm.~5.1]{BHSakuma18} are spectral results, i.e. about the support of the distribution, and do not give information about the distributions themselves. 

\subsection{Boolean Independence}\label{sec:boolean independence}

Here we first review some basic concepts of Boolean independence, introduced in \cite{SW97}. 

For a given non-commutative probability space $(\cA,\phi)$, (non-unital) subalgebras $(\cA_i)_{i\in I}$ are \emph{Bool\-ean independent} if for all $n\geq 0$, $a_1,\dots,a_n\in\cA$ are such that $a_j\in \cA_{i_j}$ where $i_1,\dots,i_n\in I$, $i_1\neq \cdots \neq i_n$, then 
\begin{equation}\label{eq:boolean_independence}
\phi(a_1\cdots a_n)=\phi(a_1)\cdots \phi(a_n).
\end{equation}
The elements $\{a_i\in\mathcal{A}\mid i\in I\}$ are Boolean independent if the subalgebras $\alg(a_i)$ generated by $a_i\ (i\in I)$ are Boolean independent. 

For $n\geq 1$, a partition $\pi$ is an \emph{interval partition} if each block of $\pi$ is an interval of the form $\{i, i+1, \dots, i+s\}$ for some $1\leq i\leq n$ and $0\leq s\leq n-i$. $I(n)$ denotes the set of all interval partitions. Then the \emph{Boolean cumulants} $\{\beta_n:\cA^n\to\bC\}_{n\geq 1}$ are defined inductively in terms of moments via
\begin{equation}\label{eq:boolean cumulant}
\phi(a_1\cdots a_n)=\sum_{\pi\in I(n)}\beta_{\pi}(a_1,\dots,a_n).
\end{equation}
As in the free case, Boolean cumulants characterize Boolean independence. To be precise, $(\cA_i)_{i\in I}$ are Boolean independent if and only if for all $n\geq 2$ and for all $a_1, \dots, a_n  \in \cup_i \cA_i$, we have $\beta_n(a_1,\dots,a_n)=0$, whenever there are at least two elements from different subalgebras; i.e. there exists $1\leq k, l \leq n$ with $i_k\neq i_l$ and $a_k \in \cA_{i_k}$ and $a_l \in \cA_{i_l}$.

The relation between moments and Boolean cumulants can also be understood via the moment-generating function and $\eta$-transform. To be precise, for a given $a\in\mathcal{A}$, the \emph{moment generating function} $\psi_a$ and the \emph{$\eta$-transform} $\eta_a$ are defined by
$$
\psi_a(z)=\sum_{n=1}^\infty m_n(a)z^n
\text{ and }
\eta_a(z)=\sum_{n=1}^\infty \beta_n(a)z^n 
$$
where $\beta_n(a)=\beta_n(a,\dots,a)$. Then we have 
\begin{equation}\label{eq:eta_psi_relation}
\eta_a (z)(1+\psi_a(z))=\psi_a(z).
\end{equation}

\subsection{Infinitesimal Boolean Independence}

The notion of infinitesimal Boolean independence was introduced in \cite{PT21}. We recall the definition of it as follows.

\begin{defn}
Suppose that $(\cA,\varphi,\varphi')$ is an infinitesimal probability space. Non-unital subalgebras $(\cA_i)_{i\in I}$ are infinitesimally Boolean independent if for all $n\geq 0, a_1,\dots,a_n\in \cA$ are such that $a_j\in \cA_{i_j}$ where $i_1,\dots,i_n\in I$, $i_1\neq \cdots \neq i_n$, then 
\begin{eqnarray*}
    \varphi(a_1\cdots a_n)&=&\phi(a_1)\cdots \phi(a_n); \\
    \varphi'(a_1\cdots a_n)&=&\sum_{j=1}^n\phi(a_1)\cdots \phi(a_{j-1})\phi'(a_j)\phi(a_{j+1})\cdots \phi(a_n).
\end{eqnarray*}
\end{defn}

Note that the infinitesimal Boolean independence can be understood as  Boolean independence with amalgamation. 
To be precise, for a given infinitesimal probability space $(\cA,\phi,\phi')$, we define 
$$
\widetilde{\cA}=\Big\{\begin{bmatrix}
 a & a' \\ 0 & a   
\end{bmatrix}  \Big | a,a'\in\mathcal{A}  \Big\},\  \widetilde{\mathbb{C}}=\Big\{\begin{bmatrix}
 c & c' \\ 0 & c   
\end{bmatrix}  \Big | c,c'\in\mathbb{C}  \Big\}
$$
and then  $\widetilde{\phi}:\widetilde{\cA}\to\widetilde{\mathbb{C}}$ is defined by
$$\widetilde{\phi}\Big(\begin{bmatrix}
 a & a' \\ 0 & a   
\end{bmatrix}\Big)=\begin{bmatrix}
    \phi(a) & \phi'(a)+\phi(a') \\ 0 & \phi(a)
\end{bmatrix}.
$$
Note that $(\widetilde{\cA},\widetilde{\mathbb{C}},\widetilde{\phi})$ forms an operator-valued probability space (see \cite{TS22}).
Moreover, in \cite{PT21}, the authors showed the following result.
\begin{thm}\label{thm:Upper-Boolean In}
Non-unital subalgebras $(\mathcal{A}_i)_{i\in I}$ are infinitesimally Boolean independent with respect to $(\phi,\phi')$ if and only if $(\widetilde{\cA}_{i})_{i\in I}$ are Boolean independent with respect to $\widetilde{\phi}$. 
\end{thm}

\subsection{Monotone Independence \cite{muraki}}

Suppose that $(\cA,\phi)$ is a non-commutative probability space. Let $\cB$ be a non-unital subalgebra and $\mathcal{C}$ be a (non-unital) subalgebra of $\cA$. $(\cB,\mathcal{C})$ are \emph{monotone independent} if whenever $a_j\in \cB\cup \mathcal{C}$ for all $j$, and $a_{k-1}\in\cB$, $a_{k}\in \mathcal{C}$, and $a_{k+1}\in\cB$ for some $k$, one has
\begin{eqnarray}
\phi(a_1\cdots a_n) &=& \phi(a_1\cdots a_{k-1}\phi(a_k)a_{k+1}\cdots a_n)
\end{eqnarray}
where one of the inequalities is eliminated if $k=1$ or $k=n.$ We denote it by $\cB \prec \mathcal{C}$.


\section{The Infinitesimal Distribution of Some Polynomials}\label{sec:some polynomials}

\begin{defn}\label{def:infinitesimal}
Let $(\cA, \phi)$ be a non-commutative probability space. If $a \in \cA$ is such that $\phi(a^n) = 0$ for all $n \geq 1$, we say that $a$ is an \textit{infinitesimal operator}.
\end{defn}

We observe that if $a_1$ and $a_2$ are free elements and $a_1$ is infinitesimal, then 
\begin{eqnarray*}
\varphi(a_1a_2a_1a_2)&=& \varphi(a_1^2)\varphi(a_2)^2+\varphi(a_2^2)\varphi(a_1)^2-\varphi(a_1)^2\varphi(a_2)^2 =0
\end{eqnarray*}
since $\varphi(a_1)=\varphi(a_1)^2=0$. In fact, we can obtain general result as the following Lemma.
\begin{lem}
Suppose $a_1$ and $a_2$ are free elements in a ncps $(\mathcal{A},\varphi)$ that $a_1$ is infinitesimal. Then 
any polynomial in $a_1$ and $a_2$ is infinitesimal, provided each term contains at least one $a_1$. In particular, $a_1a_2 + a_2a_1$ and $i(a_1a_2 - a_2a_1)$ will be infinitesimal. 
\end{lem}
\begin{proof}
By linearity, it suffices to show that 
any word $a_{i_1} \cdots a_{i_n}$ is also infinitesimal where there is at least one $j$ with $i_j = 1$. Note that $\phi((a_{i_1} \cdots a_{i_n})^m)$ may be expanded, by freeness, into a sum of products of moments of $a_1$ and moments of $a_2$. Then every term in this sum will be $0$ because $a_1$ is infinitesimal. 
\end{proof}

\begin{rem}\label{rem:inf_ideal}
With this freeness assumption, infinitesimal elements form a weak kind of ideal in $\cA$. In \cite{BHSakuma18}, the infinitesimal elements were trace class operators, which form an ideal in the usual sense.

Going a little further, if we assume that $a_1$ and $a_2$ are infinitesimally free and both infinitesimal then all mixed moments $\phi'(a_{i_1} \cdots a_{i_n})$ vanish whenever $i_1, \dots, i_n$ are not all equal. We get this by writing $\phi'(a_{i_1} \cdots a_{i_n})$ as a sum of cumulants $\partial\kappa_\pi(a_{i_1}, \dots, a_{i_n})$ with $\pi \in NC(n)$, and then observing that there can be only one $\kappa'$ in $\partial\kappa_\pi(a_{i_1}, \dots, a_{i_n})$ and all other $\kappa$'s vanish.
\end{rem}

\begin{prop}\label{prop:inf operator}
Let $a$ be an infinitesimal operator, then we have $r_a(z)=z^{-2}g_a(z^{-1})$. Furthermore, $m_n'(a)=\kappa_n'(a)$ for all $n\geq 1.$ 
\end{prop}
\begin{proof}
We observe that $\varphi(a^n)=0$ implies that $G_a(z)=z^{-1}$ and then $G_a^{\langle -1\rangle}(z)=z^{-1}.$ Then we have
$$
r_a(z)=-g_a\left(G^{\langle -1\rangle}_a(z)\right)(G^{\langle -1 \rangle}_a)'(z) =\frac{1}{z^2}g_a(\frac{1}{z}), 
$$
which implies that 
$$
\sum_{n=1}^\infty \kappa'_n(a)z^{n-1}=\frac{1}{z^2}\sum_{n=1}^\infty \frac{m_n'(a)}{(z^{-1})^{n+1}}.
$$
Thus, we obtain $\kappa_n'(a)=m_n'(a)$ for all $n\geq 1$.
\end{proof}

\subsection{The Anticommutator Case}

Let $p = x_1 x_2 + x_2 x_1$ where $x_1$ and $x_2$ are infinitesimally free and $x_2$ is infinitesimal. We shall find the infinitesimal moments of $p$ by tracking the corresponding $\varphi'(x_{i_1}\cdots x_{i_{2n}})$, and we summarize them in the following theorem. Note that this shows that $p$ is the sum of two infinitesimally freely independent copies of $x_2$, scaled by $\alpha$ and $\beta$ respectively. 

Before we state the theorem, let us introduce the following notation:
\begin{nota}
Given a string $(j_1,\dots,j_k)\in [n]^k$, this can be read as a map $j:[k]\to [n]$ by $j(l)=j_l$. Then $ker(j)$ is the partition on $[k]$ such that $r,s$ are in the same block of $ker(j)$ if and only if $j_r=j_s.$ For instance, $(j_1,j_2,j_3,j_4)=(1,2,1,2)$, then $ker(j)=\{(1,3),(2,4)\}.$
\end{nota}

\begin{thm}\label{Thm:anticommutator-moments}
Suppose that $x_1$ and $x_2$ are infinitesimally free with  $x_2$ an infinitesimal operator. Let $p = x_1x_2 + x_2x_1$ and $\alpha=m_1(x_1)+\sqrt{m_2(x_1)}$ and $\beta=m_1(x_1)-\sqrt{m_2(x_1)}$. Then $p$ is an infinitesimal operator with infinitesimal cumulants and moments 
\[
\kappa'_n(p) = m_n'(p)=\kappa_n'(\alpha x_2) + \kappa_n'(\beta x_2).
\]
Moreover, $g_{p}=g_{\alpha x_2}+g_{\beta x_2}$ and $r_{p}=r_{\alpha x_2}+r_{\beta x_2}$. 
\end{thm}

\begin{proof}
We have 
\[
\phi'(p^n) =
\sum_{j_1, \dots, j_{2n} = 1}^2
\phi'(x_{j_1} x_{j_2} \cdots x_{j_{2n-1}} x_{j_{2n}})
=
\sum_{\pi \in NC(2n)} \sum_{j_1, \dots, j_{2n} = 1}^2 
\partial \kappa_\pi(x_{j_1}, x_{j_2}, \dots, x_{j_{2n-1}}, x_{j_{2n}})
\]
where the sum over $j_1, \dots, j_{2n}$ only includes those tuples such that for each $1 \leq l \leq n$ we have $(j_{2l-1}, j_{2l})$ is either $(1, 2)$ or $(2, 1)$. Then by our assumption on $x_1$ and $x_2$, $\partial \kappa_{\pi}(x_{j_1}, \cdots, x_{j_{2n}}) = 0$ unless 
\begin{itemize}
    \item[\mybullet] $\pi\leq \ker(j)$; (\textit{by infinitesimal freeness} of $x_1$ and $x_2$)
    \item[\mybullet] $j^{-1}(2)$ is exactly one block of $\pi$ (\textit{by our assumption that $x_2$ is infinitesimal}).
\end{itemize}
For instance, when $n = 5$ and $(j_1, \dots, j_{1}) = (1,2,2,1,1,2,2,1,1,2)$

\begin{center}\begin{tikzpicture}[anchor=base, baseline]
\node[above] at (-0.75,0.25) {$\pi = \mbox{}$};
\node[above] at (0.00, 0.75) {$x_1$};
\node[above] at (0.50, 0.75) {$x_2$};
\node[above] at (1.00, 0.75) {$x_2$};
\node[above] at (1.50, 0.75) {$x_1$};
\node[above] at (2.00, 0.75) {$x_1$};
\node[above] at (2.50, 0.75) {$x_2$};
\node[above] at (3.00, 0.75) {$x_2$};
\node[above] at (3.50, 0.75) {$x_1$};
\node[above] at (4.00, 0.75) {$x_1$};
\node[above] at (4.50, 0.75) {$x_2$};
\draw  (0.00, 0.75) -- (0.00, 0.25);
\draw  (0.50 , 0.75) -- (0.50, 0.25) -- (4.50, 0.25) -- (4.50, 0.75); 
\draw  (1.00 , 0.75) -- (1.00, 0.25);
\draw  (2.50 , 0.75) -- (2.50, 0.25);
\draw  (3.00 , 0.75) -- (3.00, 0.25);
\draw  (1.50 , 0.75) -- (1.50, 0.50) -- (2.00, 0.50) -- (2.00, 0.75);
\draw  (3.50 , 0.75) -- (3.50, 0.50) -- (4.00, 0.50) -- (4.00, 0.75);
\end{tikzpicture}

\end{center}

\[
\partial \kappa_\pi (x_1, x_2, x_2, x_1, x_1, x_2, x_2, x_1, x_1, x_2) 
= \kappa_1(x_1) \kappa_2 (x_1, x_1)^2 \kappa_5'(x_2, x_2, x_2, x_2, x_2).
    \]

Then, letting $m_1 = m_1(x_1)$ and $m_2 = m_2(x_1)$,  we have for such a tuple  that
\begin{equation}\label{eq:joint_inf_moments}
\sum_{\substack{\pi \leq \ker(j)}} \partial \kappa_{\pi} (x_{j_1}, \cdots, x_{j_{2n}})
=
\kappa'_n(x_2)m_1^{n - 2v} (\sqrt{m_2})^{2v}
\end{equation}
where $0 \leq v \leq [n/2]$ counts the number of times in $j=(j_1, j_2, \dots, j_{2n})$ we have two $``1"$s cyclically adjacent. Now for a given $0 \leq v \leq [n/2]$ the number of such tuples is counted as follows. We parameterize such tuples by a string $l_1 \cdots l_n$ where $l_k = f$ (\textit{$f$ for first}) when $(j_{2k-1}, j_{2k}) = (1, 2)$ and $l_k = s$ (\textit{$s$ for second}) when $(j_{2k-1}, j_{2k}) = (2, 1)$. Then the possible strings are $f^{k_1} s^{k_2} \cdots f^{k_{r}}$ (when $r$ is odd) or $f^{k_1} s^{k_2} \cdots f^{k_{r-1}} s^{k_r}$ (when $r$ is even), where $k_1, \dots, k_r \geq 1$ and $k_1 + \cdots + k_r = n$. Here we have only to count strings that start with $f$ as, by cyclic invariance, the contribution of the strings starting with $s$ will be the same; as long as we multiply by 2 at the end. With this parametrization $v$ will be either $r/2$, when $r$ is even and $(r-1)/2$ when $r$ is odd, i.e. $v = [r/2]$. The number of $r$-tuples $(k_1, \dots, k_r)$ is $\binom{n-1}{r-1}$. Thus

\begin{align*}
\phi'(p^n)
\mbox{}=\mbox{}&
2\kappa'_n(x_2)
\sum_{r=1}^n \binom{n-1}{r-1} m_1^{n - 2[r/2]} (\sqrt{m_2})^{2[r/2]} 
= \mbox{}
\kappa'_n(x_2)\ 2 \sum_{j=0}^{[n/2]}
\binom{n}{2j} m_1^{n-2j} (\sqrt{m_2})^{2j} \\
\mbox{} = \mbox{} &
\kappa'_n(x_2) [(m_1+\sqrt{m_2})^n+(m_1-\sqrt{m_2})^n]
=
\kappa_n'(\alpha x_2) + \kappa_n'(\beta x_2).
\end{align*}
\end{proof}

In order to state a corollary, let us recall some notation from \cite{BS}. Suppose we have an infinitesimal probability space $(\cA, \varphi, \varphi')$ and $a_1, a_2 \in \cA$ with $a_1\stackrel{\cD}{\sim}(\mu_1,\mu_1')$  and $a_2 \stackrel{\cD}{\sim}(\mu_2, \mu_2')$. If $a_1$ and $a_2$ are infinitesimally free then we denote the infinitesimal distribution of $a_1 + a_2$ by $(\mu_1, \mu_1') \boxplus_{\sss B} (\mu_2, \mu_2') = (\mu, \mu')$. Then $\mu = \mu_1 \boxplus \mu_2$ is the usual free additive convolution of $\mu_1$ and $\mu_2$ , and $\mu'$ is the infinitesimal free additive convolution of $(\mu_1, \mu_1')$ and $(\mu_2, \mu_2')$.

Suppose that $x_1$ and $x_2$ are self-adjoint elements in a C$^*$-probability space that are infinitesimally free. Assume $x_2 \stackrel{\cD}{\sim}(\delta_0, \nu)$, and we let $\gamma'$ be the infinitesimal law of $p=x_1x_2+x_2x_1$.

\begin{cor}\label{Inf-g-anticomm}
$\gamma'(E)=\nu(\alpha E)+\nu(\beta E)$ for each Borel set $E$ where $\alpha=m_1(x_1)+\sqrt{m_2(x_1)}$ and $\beta=m_1(x_1)-\sqrt{m_2(x_1)}$. Equivalently,
$
g_p=g_{\alpha x_2}+g_{\beta x_2}
$ or $r_p = r_{\alpha x_2} + r_{\beta x_2}$.
\end{cor}
\begin{proof}
Recall that  $x_2 \in \cA$ in an infinitesimal operator implies that the infinitesimal $r$-transform and the infinitesimal Cauchy transform of $x_2$ are related by the simple equation $r_{x_2}(z) = z^{-2} g_{x_2}(z^{-1})$ (see Equation (\ref{eq:inf_cauchy_r_relation})). 
Then the remaining proof is immediate from Proposition \ref{prop:inf operator} and Theorem \ref{Thm:anticommutator-moments}.

\end{proof}

\subsection{The Commutator Case}

Next we consider the commutator $q =i(x_1x_2-x_2x_1)$ and compute its infinitesimal law, where we assume that $x_1$ and $x_2$ are infinitesimally free in an infinitesimal probability space 
$(\mathcal{A},\varphi,\varphi')$ and $x_2$ is an infinitesimal operator.

Our theorem  asserts that $q=i(x_1x_2-x_2x_1)$ has the distribution of the sum of infinitesimally free copies of $\omega x_2$ and $-\omega x_2$ with $\omega = \sqrt{\kappa_2(x_1)}$. To deal with the minus signs in the commutator we need a technical lemma; a proof of which is presented for the reader's convenience.

\begin{lem}\label{lemma:combinatorial sum}
For $1 \leq r \leq n$ with $n$ is even, we have
\begin{equation}\label{eq:alternating sign}
\mathop{\sum_{k_1, \dots, k_r \geq 1}}_{k_1 + \cdots + k_r = n}
(-1)^{k_2 + k_4 + \cdots + k_{2[r/2]}}
=
(-1)^{[r/2]} \binom{l-1}{[\frac{r-1}{2}]} 
\end{equation}
where $l=n/2$ we interpret the left hand side as $1$ when $r = 1$.    
\end{lem}

\begin{proof}

To find the sum (\ref{eq:alternating sign}) we let $u = k_2 + k_4 + \cdots + k_{2[r/2]}$ and condition on $u$. We first consider the case $r = 2s$ is even.

\[
\mathop{\sum_{k_1, \dots, k_r \geq 1}}_{k_1 + \cdots + k_r = n}
(-1)^{k_2 + k_4 + \cdots + k_{2[r/2]}}   
=
\sum_{u=s}^{n-s}(-1)^u \binom{u -1}{s -1} \binom{n-u-1}{s-1} 
\]
because $\ds\binom{n-u-1}{s-1} \binom{u -1}{s -1}$ is the cardinality of
\[
\{(k_1, \dots, k_r) \mid k_1, \dots, k_r \geq 1, k_1 + \cdots + k_{r-1} = n - u, k_2 + \cdots + k_r = u\}.
\]
Observe that
\[
(1-x)^{-s}(1+x)^{-s} = 
\Big[\sum\limits_{k=0}^{\infty} \tbinom{s+k-1}{s-1}x^k\Big]\cdot \Big[\sum\limits_{m=0}^{\infty}\tbinom{s+m-1}{s-1}(-x)^m\Big] 
\]
If we let $k = u - s$ and $m = n - u -s$, then $s+k-1=u-1$ and $s+m-1 = n-u-1$. Then the coefficient of $x^{n-2s}$ in $(1-x)^{-s}(1+x)^{-s}$ is 
\[
\mathop{\sum_{k,m \geq 0}}_{k + m = 2n -s}
(-1)^m\tbinom{s+k-1}{s-1} \tbinom{s+m-1}{s-1}
=
\sum_{u=s}^{n-s} (-1)^{u + s}
\binom{u -1}{s -1} \binom{n-u-1}{s-1}
\]
On the other hand, 
$\ds
(1-x)^{-s}(1+x)^{-s}= (1-x^2)^{-s}=\sum\limits_{j=0}^{\infty} \tbinom{s+j-1}{s-1}x^{2j}. 
$
Note that $2j=n-2s$ implies that $j=\frac{n}{2}-s$ (recall that $n$ is assumed to be even). Then the coefficient of $x^{n-2s}$ of $(1-x)^{-s}(1+x)^{-s}$ is $\tbinom{s+(n/2-s)-1}{s-1}=\tbinom{n/2-1}{s-1}.$
Thus, we conclude that (still assuming $r = 2s$)
\[
\mathop{\sum_{k_1, \dots, k_r \geq 1}}_{k_1 + \cdots + k_r = n}
(-1)^{k_2 + k_4 + \cdots + k_{2[r/2]}}
=
(-1)^{s} \binom{l-1}{s-1} 
=
(-1)^{[r/2]} \binom{l-1}{[\frac{r-1}{2}]}
\]
Finally assume that $r = 2 s + 1$ is odd. 
\[
\mathop{\sum_{k_1, \dots, k_r \geq 1}}_{k_1 + \cdots + k_r = n}
(-1)^{k_2 + k_4 + \cdots + k_{2[r/2]}}   
=
\sum_{u=s}^{n-(s+1)}(-1)^u \binom{u -1}{s - 1} \binom{n-u-1}{s} 
\]
because $\ds\binom{n-u-1}{(s+1)-1} \binom{u -1}{s -1}$ is the cardinality of
\[
\{(k_1, \dots, k_r) \mid k_1, \dots, 
   k_r \geq 1, k_1 + \cdots + k_{r} = n - u, 
   k_2 + \cdots + k_{r-1} = u\}.
\]
Observe that
\[
(1-x)^{-s}(1+x)^{-(s+1)} = 
\Big[\sum\limits_{k=0}^{\infty} \tbinom{s+k-1}{s-1}x^k\Big]\cdot \Big[\sum\limits_{m=0}^{\infty}\tbinom{s+m}{s}(-x)^l\Big] 
\]
If we let $k = u - s$ and $m = n - u -s - 1$, then $s+k-1=u-1$ and $s+m = n-u-1$. Then the coefficient of $x^{n - 2s -1}$ in $(1 - x)^{-s} (1 = s)^{-(s+1)}$ is 
\[
\mathop{\sum_{k,l \geq 0}}_{k + m = n -2s -1}
(-1)^l \tbinom{s+k-1}{s-1} \tbinom{s+m}{s}
=
\sum_{u=s}^{n-(s+1)} (-1)^{u + s + 1}
\binom{u -1}{s -1} \binom{n-u-1}{s}.
\]
On the other hand, 
$\ds
(1-x)^{-s}(1+x)^{-(s+1)}= (1 - x)(1-x^2)^{-(s+1)}
=
\sum\limits_{j=0}^{\infty} \tbinom{s+j}{s}x^{2j}
-
\sum\limits_{j=0}^{\infty} \tbinom{s+j}{s}x^{2j+1}. 
$
Since $n - 2s -1$ is odd, so the coefficient of $x^{n - 2 s - 1}$ in $(1 - x)^{-s} (1 + x)^{-(s + 1)} $ is $- \tbinom{s + j}{s}$ when $n - 2s -1 = 2j + 1$, i.e. $j = n/2 -(s+1)$. Thus
\[
\mathop{\sum_{k_1, \dots, k_r \geq 1}}_{k_1 + \cdots + k_r = n}
(-1)^{k_2 + k_4 + \cdots + k_{2[r/2]}} 
=
(-1)^s \binom{s + n/2 - (s+1)}{s}
=
(-1)^s \binom{l-1}{s} = (-1)^{[r/2]} \binom{l-1}{[\frac{r-1}{2}]}.
\]
\end{proof}

\begin{thm}\label{Thm:commutator-moments}
Suppose $x_1, x_2 \in (\cA, \phi, \phi')$ are infinitesimally free with $x_2$ an infinitesimal operator. Let $q=i(x_1x_2-x_2x_1)$ and  $\omega=\sqrt{\kappa_2(x_1)}$. Then the infinitesimal moments of $q$ are given by
    \begin{equation}
        \kappa_n'(q) = m_n'(q)= \left.\begin{cases}
            2\kappa_n'(\omega x_2) & \text{ if }  n \text{ is even} \\
            0  & \text{ if }  n \text{ is odd} 
        \end{cases}\right\}
        =  \kappa_n'(\omega x_2) + \kappa_n'(-\omega x_2).
    \end{equation}
Moreover, 
$
g_q=g_{\omega x_2}+g_{-\omega x_2}
$
and
$
r_q=r_{\omega x_2}+r_{-\omega x_2}.
$
\end{thm}

\begin{proof}
Writing $-i q = x_1 x_2 - x_2 x_1$ we have 
\begin{align}\label{eq:q-sum1}
\phi'((-iq)^n) &=  \kern-1em
\sum_{j_1, \dots, j_{2n} = 1}^2 \kern-0.5em (-1)^{\psi(j)}
\phi'(x_{j_1} x_{j_2} \cdots x_{j_{2n-1}} x_{j_{2n}}) \notag \\
=\mbox{}&
 \kern-0.5em\sum_{\pi \in NC(2n)} \sum_{j_1, \dots, j_{2n} = 1}^2 (-1)^{\psi(j)}
\partial \kappa_\pi(x_{j_1}, x_{j_2}, \dots, x_{j_{2n-1}}, x_{j_{2n}})
\end{align}
where the sum over $j_1, \dots, j_{2n}$ only includes those tuples such that for each $1 \leq l \leq n$ we have $(j_{2l-1}, j_{2l})$ is either $(1, 2)$; or $(2, 1)$ and $\psi(j)$ is the number of times that $(j_{2l-1}, j_{2l}) = (2,1)$. Then by our assumption on $x_1$ and $x_2$, $\partial \kappa_{\pi}(x_{j_1}, \cdots, x_{j_{2n}}) = 0$ unless 
\begin{itemize}
    \item[\mybullet] $\pi\leq \ker(j)$; (\textit{by infinitesimal freeness} of $x_1$ and $x_2$)
    \item[\mybullet] $j^{-1}(2)$ is exactly one block of $\pi$ (\textit{by our assumption that $x_2$ is infinitesimal}).
\end{itemize}
Then, letting $m_1 = m_1(x_1)$ and $m_2 = m_2(x_1)$,  we have for such a tuple  that
\begin{equation*}
\sum_{\substack{\pi \leq \ker(j)}} \partial \kappa_{\pi} (x_{j_1}, \cdots, x_{j_{2n}})
=
\kappa'_n(x_2)m_1^{n - 2v} (\sqrt{m_2})^{2v}
\end{equation*}
where $0 \leq v \leq [n/2]$ counts the number of times in $j=(j_1, j_2, \dots, j_{2n})$ we have two $``1"$s cyclically adjacent. Now for a given $0 \leq v \leq [n/2]$ the set of such tuples is parameterized as follows. We parameterize such tuples by a string $l_1 \cdots l_n$ where $l_k = f$ (\textit{$f$ for first}) when $(j_{2k-1}, j_{2k}) = (1, 2)$ and $l_k = s$ (\textit{$s$ for second}) when $(j_{2k-1}, j_{2k}) = (2, 1)$. Then the possible strings are $f^{k_1} s^{k_2} \cdots f^{k_{r}}$ and $s^{k_1} f^{k_2} \cdots s^{k_{r}}$ (when $r$ is odd) or $f^{k_1} s^{k_2} \cdots f^{k_{r-1}} s^{k_r}$ and $s^{k_1} f^{k_2} \cdots s^{k_{r-1}} f^{k_r}$ (when $r$ is even), where $k_1, \dots, k_r \geq 1$ and $k_1 + \cdots + k_r = n$. The number of $r$-tuples $(k_1,\dots,k_r)$ is $\tbinom{n-1}{r-1}$. In each of these cases (depending on the parity of $r$), $v$ is the same; indeed with this parametrization $v$ will be either $r/2$, when $r$ is even and $(r-1)/2$ when $r$ is odd, i.e. $v = [r/2]$. Thus replacing $f$ by $s$ and $s$ by $f$ does not change the contribution, but this switch is equivalent to replacing $q$ by $-q$. Hence $\phi'(q^n) = -\phi'(q^n)$ for $n$ odd. Thus the odd moments are $0$ and from now on we shall assume that $n$ is even.

Since $n$ is even we have $(-1)^{k_1 + k_3+ \cdots + k_{r-1}} = (-1)^{k_2 + k_4 + \cdots + k_r}$ (when $r$ is even) and $(-1)^{k_1 + \cdots + k_r} = (-1)^{k_2 + \cdots + k_{r-1}}$ (when $r$ is odd). Hence $(-1)^{\psi(j)}$ is unchanged when we replace $f$ by $s$. Thus we have only to count strings that start with $f$; as long as we multiply by 2 at the end. 

Now assuming $n = 2 l$ is even, we rewrite the sum (\ref{eq:q-sum1}), conditioning on $r$:
\[
\phi'((-iq)^n)
=
2\kappa_n'(x_2)\sum_{r=1}^n m_1^{n - 2[r/2]}(\sqrt{m_2})^{2[r/2]} \kern-1em
\mathop{\sum_{k_1, \dots, k_r \geq 1}}_{k_1 + \cdots + k_r = n} \kern-1em
(-1)^{k_2 + k_4 + \cdots + k_{2[r/2]}},
\]
where we set $\kern-1em\ds\mathop{\sum_{k_1, \dots, k_r \geq 1}}_{k_1 + \cdots + k_r = n} \kern-1em
(-1)^{k_2 + k_4 + \cdots + k_{2[r/2]}} =1$ when $r = 1$. 
By Lemma \ref{lemma:combinatorial sum} we have that
\[
\phi'((-iq)^n)
=
2 \kappa_n'(x_2)
\sum_{s = 0}^{l}
\binom{l }{s} (m_1^2)^{l -s} (-m_2)^{s} 
=
2 \kappa_n'(x_2) (m_1^2 - m_2)^l
=
2\kappa_n'(\omega x_2)
\]
because $\binom{l-1}{[\frac{r-1}{2}]} + \binom{l-1}{[\frac{r}{2}]} = \binom{l}{s}$ when $r = 2s$.
\end{proof}

\begin{rem}\label{Inf-Cauchy-commutator}
In the anticommutator case with $x_2\stackrel{\cD}{\sim} (\delta_0,\nu)$ for $\nu$ is a signed measure on $\bR$, we have, by Stieltjes inversion, $p \stackrel{\cD}{\sim} (\delta_0, \gamma_p)$ where $\gamma_p(E)=\nu(\alpha E)+\nu(\beta E)$ for each Borel set $E \subseteq \bR$ and $\alpha=m_1(x_1)+\sqrt{m_2(x_1)}$ and $\beta=m_1(x_1)-\sqrt{m_2(x_1)}$. In the commutator case we have $q \stackrel{\cD}{\sim} (\delta_0, \gamma_q)$ where $\gamma_q(E)=\nu(\omega E)+\nu(-\omega E)$ for each Borel set $E \subseteq \bR$ and $\omega=\sqrt{\kappa_2(x_1)}$. Moreover the distribution of $q$ depends only on $\nu$ and the second moment of $x_1$ as centering $x_1$ does not change $q$. In addition we see that the infinitesimal distribution of both $p$ and $q$ is the distribution of a sum of  infinitesimally freely independent copies of $x_2$, suitably rescaled.
\end{rem}

\begin{rem}
Suppose $x = b_1 a_1 b_1^* + \cdots + b_k a_k b_k^*$ with
$a_1, \dots, a_k$ infinitesimal operators in $(\cA, \phi,
\phi')$ infinitesimally free from $\{b_1, \dots, b_k\}$. Let
$A = \diag(a_1, \dots, a_k )$, $B = (\beta_{ij})_{i,j =
  1}^k$ with $\beta_{ij} = \phi(b_i b_j^*)$, and $X = \sqrt{B}A\sqrt{B} \in M_k(\cA)$. Then
\[
\phi'(x^n)
= \Tr( K'_n(AB, \dots, AB))
= \Tr( K'_n(\sqrt{B}A\sqrt{B}, \dots, \sqrt{B}A\sqrt{B}))
= \Tr \otimes \phi'(X^n),
\]
where $K'_n(AB, \dots, AB)$ is the $n^{th}$  $M_k(\bC)$-valued cumulant of $AB$ and the last equality follows from Proposition \ref{prop:inf operator}. Hence $x$ and $X$ have the same distribution.
\end{rem}  

\begin{rem}
Note that the infinitesimal law of anti-commutators and commutators can also be computed via the linearization trick. Recall that for a given a  polynomial $p\in\mathbb{C}\langle x_1,\dots, x_n\rangle$ in non-commuting variables $x_1, \dots, x_n$, a matrix $L_p\in M_m(\mathbb{C}\langle x_1,\dots, x_n\rangle)$,
 is called a \emph{linearization of $p$} if there are $b_0,\dots,b_n\in M_m(\mathbb{C})$, $u\in M_{m-1\times 1}(\mathbb{C}\langle x_1,\dots, x_n \rangle)$,  $v\in  M_{1\times m-1}(\mathbb{C}\langle x_1,\dots, x_n \rangle)$ and $Q\in M_{m-1}(\mathbb{C}\langle x_1,\ab\dots,x_n\rangle)$ such that
 \[L_p =\begin{bmatrix} 0 & u \\ v & Q \end{bmatrix} =b_0\otimes 1 + b_1\otimes x_1+\cdots + b_n \otimes x_n,
 \mbox{ with } Q \mbox{ invertible, and } p = -uQ^{-1}v.
 \] 

 A linearization can always be found, see \cite[Ch.~9]{MS} and \cite{BMS}. Given a linearization, let 

 $$X  = \begin{bmatrix} 0 & -uQ^{-1} \\v & 0\end{bmatrix}; \mbox{ then }
 X^2  = \begin{bmatrix} -uQ^{-1}v & 0 \\0  -vuQ^{-1}\end{bmatrix}.
      $$ 

This leads to the following result:
\begin{equation}\label{eqn: linearization-moment}
p^n = [X^{2n}]_{1,1} \textrm{, thus } \varphi'(p^n) = \left[E'(X^{2n})\right]_{1,1}\text{ for all }n\in\mathbb{N}.
\end{equation}
 where $[A]_{1,1}$ denotes the $(1,1)$-entry of the matrix $A$ and $E' = \phi' \otimes Id$. In the case of the commutator and anticommutator the matrices $b_0$, $b_1$, and $b_2$ are very simple; so one could obtain a proof of Theorems \ref{Thm:anticommutator-moments} and \ref{Thm:commutator-moments} in this way. However we found the proof presented here simpler in these special cases. 
\end{rem}

\subsection{Random Matrix Examples}

Note that there are now many results which show that finite rank matrices are asymptotically infinitesimally free from standard random matrix models, see Au \cite{Au21,Au23}, Dallaporta and Février \cite{df}, Collins, Hasebe, and Sakuma \cite{BHSakuma18}, Collins, Leid, and Sakuma \cite{cls}, and Shlyakhtenko \cite{shly}. This gives the results of this section a wide application.

In \cite[Thm.~5.1]{BHSakuma18} Collins, Hasebe and Sakuma considered the limit spectrum of the commutator and anticommutator: $i(A_N U_N B_N U_N^* - U_N B_N U_N^* A_N)$ and $A_N U_N B_N U_N^* + U_N B_N U_N^* A_N$ respectively. Now we show that their results are a consequence of our Theorem \ref{Thm:anticommutator-moments}
and Theorem \ref{Thm:commutator-moments}. If $a$ is infinitesimal and in addition infinitesimally free from $b$ then as observed in Remark \ref{Inf-Cauchy-commutator},  when $\alpha=m_1(b)+\sqrt{m_2(b)}$, $\beta=m_1(b)-\sqrt{m_2(b)}$ and $\omega=\sqrt{\kappa_2(b)}$, we have 
\begin{equation*}\label{g_combined}
g_{ab+ba}(z)=g_{\alpha a}(z)+g_{\beta a}(z) \mathrm{\ and \ } g_{i(ab-ba)}(z)= g_{\omega a}(z)+g_{-\omega a}(z), \mbox{\ and\ }
\end{equation*}
\begin{equation}\label{r_combined}
r_{ab+ba}(z)=r_{\alpha a}(z)+r_{\beta a}(z) \mathrm{\ and \ } r_{i(ab-ba)}(z)= r_{\omega a}(z)+r_{-\omega a}(z).
\end{equation}
The first set of equations gives us the spectral results of \cite{BHSakuma18} and the second equation gives us that the distribution is the free additive convolution of two free copies of $a$, after rescaling.


\section{Infinitesimal $R$-diagonal Operators}\label{sec:R-diagonal}

In this Section, we extend $R$-diagonal properties (see Subsection \ref{Pre: R-diagonal}) to the infinitesimal setting in the obvious way.

Recall that $(\cA,\phi,\phi')$ is a \textit{infinitesimal $*$-probability space} if $(\cA,\varphi)$ is a $*$-probability space with $\varphi'$ is self-adjoint in the sense that $\varphi'(a^*)=\overline{\varphi'(a)}$ for all $a\in\cA.$
\begin{defn}
Let $(\mathcal{A},\varphi,\varphi')$ be a infinitesimal $*$-probability space.  We say $a\in\mathcal{A}$ is \emph{infinitesimal $R$-diagonal} if it is $R$-diagonal and in addition for all $\epsilon_1, \dots , \epsilon_n \in \{-1, 1\}$, we have that 
\[
\kappa_n'(a^{(\epsilon_1)},\dots,a^{(\epsilon_n)})=0 
\]
unless $n$ is even and $\epsilon_i + \epsilon_{i+1} = 0$ for $1 \leq i < n$.

We say that $a \in \cA$ is \textit{infinitesimally even} if $a$ is even and $\phi'(a^{2n+1})=0$ for all $n\in\mathbb{N}$.
\end{defn}

\begin{prop}
Suppose $(\mathcal{A},\varphi,\varphi')$ is a $*$-infinitesimal probability space, and $x_1=x_1^*$ and $x_2=x_2^*$ are infinitesimally free elements in $\mathcal{A}$ such that  $x_2$ is an infinitesimal operator. We further assume that $m_1(x_1) = 0$, i.e. $x_1$ is centered. Then $c=x_1x_2$ is infinitesimally $R$-diagonal, and
\[
\kappa_{n}'(c, c^*, c, c^*,\dots, c, c^*)=\kappa_n'(c^*, c, c^*, c,\dots, c^*,c)= \kappa_n'(x_2)\kappa_2(x_1)^{n/2}.
\]
\end{prop}

\begin{proof}
We first observe that the $*$-distribution of $c$ is 0, by this we mean that all $*$-moments $\phi(c^{(\epsilon_1)}\cdots c^{(\epsilon_n)})$ equal $0$. Indeed, all cumulants of $x_2$ are 0, and every $*$-cumulant of $c$ can be written as a linear combination of products of cumulants of $x_1$ and cumulants of $x_2$ using the formula for cumulants with products as entries, c.f. \cite[Thm.~11.12]{NS}, and the vanishing of mixed cumulants, c.f. \cite[Thm.~11.16]{NS}. Thus all $*$-cumulants of $c$ are 0, and in particular $c$ is $R$-diagonal.

Now suppose that $\epsilon_1, \dots, \epsilon_k \in \{-1, 1\}$; we must now show that  
$
\kappa_n'(c^{(\epsilon_1)},\dots, c^{(\epsilon_k)})=0 
$
unless $k$ is even and $\epsilon_i + \epsilon_{i+1} = 0$ for $1 \leq i < k$. For the purpose of this part of the proof let us set $(j_{2l-1}, j_{2l}) = (1, 2)$ if $\epsilon_l = 1$ and $(j_{2l-1}, j_{2l}) = (2, 1)$ if $\epsilon_l = -1$. Let $\delta = \{(1, 2), \dots, (2k-1, 2k)\}$.  Again, by the formula for cumulants with products as entries  (this time in the infinitesimal case, see (\cite[Prop.~3.15]{FN}), we have 
\[
\kappa_n'(c^{(\epsilon_1)},\dots, c^{(\epsilon_n)})
=
\mathop{\sum_{\pi \in NC(2k)}}_{\pi \vee \delta = 1_{2k}}
\partial \kappa_\pi(x_{j_1}, x_{j_2}, \dots, x_{j_{2k-1}}, x_{j_{2k}})
\]
After the discussion following Equation (\ref{eq:joint_inf_moments}), we observe that $\partial \kappa_{\pi}(x_{j_1},x_{j_2},\dots,x_{j_{2k-1}},x_{j_{2k}})=0$ unless
$j^{-1}(2)$ is a block of $\pi$ of size $k$ and all remaining blocks of $\pi$  are either a pair connecting one $x_1$ to another $x_1$, or a singleton with a single $x_1$. Since we have assumed that $m_1(x_1) = 0$, all the $x_1$-blocks of $\pi$ must have size $2$, or else the contribution vanishes. This can only happen if each $c$ is followed cyclically by a $c^*$, and hence $k$ must be even and the $\epsilon$'s must alternate. This shows that $c$ is infinitesimally $R$-diagonal. Moreover,  there is only one such $\pi$ and for this $\pi$ we have $\partial \kappa_{\pi}(x_{j_1},x_{j_2},\dots,x_{j_{2k-1}},x_{j_{2k}})= \kappa_k'(x_2) \kappa_2(x_1)^{k/2}$. Hence 
\[
\kappa_{n}'(c,c^*,c,c^*,\dots, c, c^*)=\kappa_n'(c^*, c, c^*, c,\dots, c^*, c)= \kappa_n'(x_2)\kappa_2(x_1)^{n/2}
\]
as claimed. 
\end{proof}

Recall that if $a$ is an $R$-diagonal element in a $*$-probability space $(\mathcal{A},\varphi)$, then the $*$-distribution of $a^*a$ and $aa^*$ can be uniquely determined by $(\alpha_n)_n$ and $(\beta_n)_n$ where
$$
\alpha_n=\kappa_{2n}(a,a^*,a,a^*,\dots, a,  a^*) \text{ and }\beta_n=\kappa_{2n}(a^*,a,a^*,a\cdots, a^*,a).
$$
Here we will generalize this idea to the infinitesimal free framework. For a given element $a\in\mathcal{A}$, we denote 
$$
\alpha'_n=\kappa_{2n}'(a,a^*,a,a^*,\cdots, a, a^*) \text{ and }\beta'_{2n}=\kappa'_{2n}(a^*,a,a^*,a,\cdots, a^*, a). 
$$ 
Let us first observe that we have the following Lemma immediately from \cite[Prop.~3.15]{FN}.

\begin{lem}\label{lemma:Icumulant of aa^*}
Fix $n\in\mathbb{N}$, let $\sigma=\{(1,2),(3,4),\dots,(2n-1,2n)\}$. Then we have
$$
    \kappa_n'(aa^*,\dots, aa^*)=\sum_{\substack{\pi\in NC(2n)\\
                  \pi \vee \sigma=1_{2n}}} 
                  \partial \kappa_{\pi}(a,a^*,\dots, a, a^*) 
$$
and
$$
\kappa_n'(a^*a,\dots,a^*a)=\sum_{\substack{\pi\in NC(2n)\\
                  \pi \vee \sigma=1_{2n}}} 
                  \partial \kappa_{\pi}(a^*,a,\dots,a^*,a).
$$
\end{lem}

\begin{prop}
Let $a$ be an infinitesimal $R$-diagonal element. Then we have
\begin{equation}\label{prop:Icumulant of aa^*}
\kappa'_n(aa^*,\dots, aa^*) = \sum_{\substack{\pi\in NC(n)\\
                  \pi=\{V_1,\dots,V_r\}}} \left(\alpha'_{|V_1|}\prod_{j=2}^r\beta_{|V_j|}+\alpha_{|V_1|}\sum_{j=2}^r\beta_{|V_j|}'\prod_{i=2,i\neq j}^r\beta_{|V_i|}\right)
\end{equation}
and
\begin{equation}\label{prop:Icumulant of a^*a}
\kappa_n'(a^*a,\dots,a^*a)=\sum_{\substack{\pi\in NC(n)\\
                  \pi=\{V_1,\dots,V_r\}}}                 \left(\beta'_{|V_1|}\prod_{j=2}^r \alpha_{|V_j|}+\beta_{|V_1|}\sum_{j=2}^r\alpha_{|V_j|}'\prod_{i=2,i\neq j}^r\alpha_{|V_i|}\right).
\end{equation}
where $V_1$ is the block of $\pi$ which contains the first element $1$. 
\end{prop}

\begin{proof}
Following Lemma \ref{lemma:Icumulant of aa^*}, we have
$$
 \kappa_n'(aa^*,\dots,aa^*)=\sum_{\substack{\pi\in NC(2n)\\
                  \pi \vee \sigma=1_{2n}}} 
                  \partial \kappa_{\pi}(a,a^*,\dots, a, a^*). 
$$
where $\sigma=\{(1,2),\dots,(2n-1,2n)\}$. We note that if $\pi\in NC(2n)$ such that $\pi\vee \sigma=1_{2n}$, then $\pi$ is canonical bijection with $NC(n)$ (see  \cite[Prop.~11.25]{NS}). Thus \eqref{prop:Icumulant of aa^*} can be deduced immediately from the fact that 
$$
\partial\kappa_{\pi}(a,a^*,\dots,a,a^*)= \alpha'_{|V_1|}\prod_{j=2}^r\beta_{|V_j|}+\alpha_{|V_1|}\sum_{j=2}^r\beta_{|V_j|}'\prod_{i=2,i\neq j}^r\beta_{|V_i|}.$$
The proof of \eqref{prop:Icumulant of a^*a} is the same as above. 
\end{proof}
By applying Lemma \ref{lemma:Icumulant of aa^*} and combining the vanishing of mixed infinitesimal free cumulants property with the analogous argument in the proof of \cite[Prop.~15.8]{NS}, we have the following result. 

\begin{prop}
Let $a$ and $b$ are elements in a infinitesimal $*$-probability space that $a$ and $b$ are $*$-infinitesimally free. Suppose that $a$ is infinitesimal $R$-diagonal, then $ab$ is also infinitesimal $R$-diagonal.
\end{prop}


\section{Infinitesimal Idempotents and  Boolean Cumulants}\label{sec:idempotents}

In this section we consider the special case when our infinitesimal operator is an idempotent. We get the surprising result that Boolean cumulants can be expressed as infinitesimal moments. Theorem \ref{thm:boolean independence} means that any random matrix model that produces asymptotic infinitesimal freeness (see \cite{shly} and \cite{df}) will automatically produce asymptotic Boolean independence. Theorem \ref{thm:monotone independence} will accomplish the same for monotone independence. In particular, this allows us to recover the asymptotic Boolean independence results of Male \cite[Cor.~5.12]{male} and Lenczewski \cite[Thm.~7.1]{l}. 
Note that there are now many results which show that finite rank matrices are asymptotically infinitesimally free from standard random matrix models, see Au \cite{Au21}, Dallaporta and Février \cite{df}, Collins, Hasebe, and Sakuma \cite{BHSakuma18}, Collins, Leid, and Sakuma \cite{cls}, and Shlyakhtenko \cite{shly}. This gives the results of this section a wide application. An example is provided in \cite[\S 5]{BT23}.
\begin{nota}
Suppose that we have an infinitesimal probability space $(\cA, \phi, \phi')$ and $j \in \cA$, an infinitesimal idempotent with $\phi'(j) \not = 0$.  Suppose further we have a unital subalgebra $\cB \subseteq \cA$ such that $\{ j\}$ and $\cB$ are infinitesimally free with respect to $(\phi, \phi')$. 
We let $\psi(a) = \phi'(aj)/\phi'(j)$. Then $\psi(1) = 1$ and we may consider the new non-commutative probability space $(\cA, \psi)$. Note that for $a \in \cB$ we have 
\[
\psi(a) = \phi'(aj)/\phi'(j) = \phi'(j)\phi(a)/\phi'(j) = \phi(a).
\]
Let $j^{(-1)} = j$ and $j^{(1)} = \jp:=1-j$.  We define $\mathcal{J}(\cB)$ to be the algebra  generated by
$$
\{j^{(\varepsilon_1)}a_1j^{(\varepsilon_2)}a_2\cdots j^{(\varepsilon_{k-1})}a_{k-1}j^{(\varepsilon_k)}\mid k\geq 0,\ \varepsilon_1,\dots,\varepsilon_k\in \{-1,1\}, \text{ and }a_1,a_2,\dots,a_k\in\cB \},
$$
and $\mathcal{J}_a(\cB)$ be the subalgebra of $\mathcal{J}(\cB)$ where we require 
$\varepsilon_1\neq \varepsilon_2 \neq \cdots \neq \varepsilon_{k-1}\neq \varepsilon_k$.
\end{nota}

\begin{nota}\label{nota:epsilon notation}
Given a string $\epsilon_1, \epsilon_2, \dots, \epsilon_n \in \{-1, 1\}$ with $\epsilon_1 = -1$, we let $\sigma_\epsilon \in I(n)$ be the interval partition where each $l$ such that $\epsilon_l = -1$ is the start of a new block. For example, suppose $n=9$ and $\epsilon_1 = \epsilon_2 = -1, \epsilon_3 =  \epsilon_4 = 1, \epsilon_5 =  \epsilon_6 = -1, \epsilon_7 = \epsilon_8 = \epsilon_9 =1$; then
$
\sigma_\epsilon = \{ (1),(2,3,4),(5),(6, 7, 8, 9)\}.
$
Given two strings  $(\epsilon_1, \dots, \epsilon_n)$ and $(\eta_1, \dots, \eta_n)$ in $\{-1, 1\}^{n}$, we say that $\epsilon \leq \eta$ if $\epsilon_k \leq \eta_k$ for $k = 1, \dots, n$. Then $\epsilon \leq \eta$ $\Leftrightarrow$ $\sigma_\epsilon \leq \sigma_\eta$. The strings $(\epsilon_1, \dots, \epsilon_n)$ parametrize $I(n)$.

Next suppose we have a string $\epsilon_1, \dots, \epsilon_n \in \{-1, 1\}$ with $\epsilon_1 = 1$; we then create a non-crossing partition $\sigma_\epsilon \in \NC(n)$ in which the blocks are cyclic intervals. By a \textit{cyclic interval} we mean one where we consider $n$ and $1$ to be adjacent (see \cite[Figure 5]{ahl}.)  To obtain $\sigma_\epsilon$, we  start a new interval at $l$ whenever $\epsilon_l = -1$. If $|\{ l \mid \epsilon_l = -1\}| \leq 1$ then $\sigma_\epsilon = 1_n$, otherwise  we will have $\sigma_\epsilon \not\in I(n)$. So we let $\CI(n)$ (as in \cite{ahl}) denote the set of the cyclic interval partitions, i.e. non-crossing partitions which when viewed on the circle are interval partitions.  If we have $a_1, \dots, a_8 \in \cA$  and  $(\epsilon_1, \epsilon_2, \epsilon_3, \epsilon_4, \epsilon_5, \epsilon_6, \epsilon_7, \epsilon_8) = (1, -1, -1, -1, 1, 1, -1, -1) $, then $\sigma_{\epsilon} =\big\{ (1,8),(2),(3), (4,5,6),(7)\big\}$ with 
\[
\beta_{\sigma_\epsilon}(a_1, a_2, a_3, a_4, a_5, a_6, a_7, a_8) = \beta_2(a_1, a_8) \beta_1(a_2) \beta_1(a_3) \beta_3(a_4, a_5, a_6) \beta_1(a_7).
\] 

This extends the notation of Boolean cumulants mentioned in  equation (\ref{eq:boolean cumulant}) in \S \ref{sec:boolean independence}. We define the following Boolean-like cumulant using cyclic interval partitions. 

\[
\tilde\beta_n(a_1, \dots, a_n)
\mbox{}:=\mbox{} (n-1) \, \phi(a_1 \cdots a_n) \,
- \kern-0.5em \sum_{\sigma \in \CI(n)} (-1)^{\#(\sigma)} \phi_\sigma(a_1, \dots, a_n).
\]
This is very similar to the definition of Boolean cumulants except we allow cyclic intervals and we attach weight $n$ to the partition $1_n$ instead of $1$, which is what the formula for the other partitions would produce. 
\end{nota}

\begin{prop}\label{prop:boolean cumulants}
Suppose $(\cA, \phi, \phi')$ is an infinitesimal probability space, $j \in \cA$ is an infinitesimal idempotent with $\phi'(j) \not = 0$, and $\cB \subseteq \cA$ is a unital subalgebra infinitesimally free from $j$. 
Then for all $a_1, \dots, a_n \in \cB$ we have
\begin{enumerate}\setlength{\itemsep}{0.5em}%

\item \label{prop:boolean cumulantsi}
$\ds
\phi'(j a_1 j a_2 j \cdots j a_n ) 
=
\phi'(j)\phi(a_1) \phi(a_2) \cdots \phi(a_n)
=
\phi'(j a_1 j a_2 j \cdots j a_n j)
$

and
$\ds
\phi'( a_1 j a_2 j \cdots j a_n ) \allowbreak
=
\phi'(j)\phi(a_1a_n) \phi(a_2) \cdots \phi(a_{n-1})
$.

\item \label{prop:boolean cumulantsii}
for any $\epsilon_1, \dots, \epsilon_n \in \{-1, 1\}$ with $|\{ l \mid \epsilon_l = -1\}| \geq 1$ let $\sigma_\epsilon$ be the partition in Notation \ref{nota:epsilon notation}, then

$\ds \phi'(j^{(\epsilon_1)} a_1 j^{(\epsilon_2)}a_2 \cdots  j^{(\epsilon_n)}a_n) = \phi'(j)\beta_{\sigma_\epsilon}(a_1, \dots, a_n)$.

\item \label{prop:boolean cumulantsii.v}
if $\epsilon_1 =-1$ then  $\sigma_\epsilon \in I(n)$ and $\ds \phi'(j^{(\epsilon_1)} a_1 j^{(\epsilon_2)}a_2 \cdots  j^{(\epsilon_n)}a_nj) = \phi'(j) \beta_{\sigma_\epsilon}(a_1, \dots, a_n)$.

\item \label{prop:boolean cumulantsiii}
$\ds
\phi'(j^\perp a_1 \jp a_2 \jp a_3 \jp \cdots \jp a_n)
=
\phi'(a_1 \cdots a_n) - \phi'(j)\tilde\beta_n(a_1, \dots, a_n).
$

\end{enumerate}

\end{prop}

\begin{proof}
By the infinitesimal moment cumulant formula we have
\[
\phi'(j a_1 j a_2 j \cdots j a_n )
=
\sum_{\pi \in NC(2n)} 
\partial \kappa_\pi (j, a_1, j, a_2, j, \dots, j, a_n ).
\]
We have $\partial \kappa_\pi(j, a_1, j, a_2, j, \dots, j, a_n ) = 0$ unless $\pi = \{(1, 3, \dots, 2n-1) ,(2), (4), \dots, (2n)\}$, because we cannot have an $a_i$ and a $j$ in the same block (by infinitesimal freeness) and we can have only one block for the $j$'s (by the infinitesimality of $j$). Thus 
\begin{align*}
\phi'&(ja_1j a_2 j \cdots j a_n)
=
\kappa'_{n}(j, \dots, j) \phi(a_1) \cdots \phi(a_n)  =
\phi'(j)   \phi(a_1) \cdots \phi(a_n).
\end{align*}
This proves the first equality in the first part of  $($\ref{prop:boolean cumulantsi}$)$. For the second equality we must have $\pi = \{(1, 3, \dots, 2n+1),(2),(4), \dots, (2n)\}$. 

For the second part of $($\ref{prop:boolean cumulantsi}$)$ we must have $\pi = \{(2, 4, \dots, 2n-2) ,(1), (3), \dots, (2n-1)\}$ in order to get a non-zero contribution, in which case we get 
\begin{align*}
\phi'&(a_1j a_2 j \cdots j a_n )
=
\kappa'_{n-1}(j, \dots, j) \phi(a_1a_n) \cdots \phi(a_{n-1})  =
\phi'(j)   \phi(a_1a_n) \cdots \phi(a_{n-1}).
\end{align*}

\medskip

For the proof of $($\ref{prop:boolean cumulantsii}$)$ we use the following notation. Let $\lambda_{-1} = j$ and $\lambda_{1}  = 1$, then $\jp = 1-j= \lambda_{1}-\lambda_{-1} $. Then 
\begin{align*}
\phi'(j^{\epsilon_1)} a_1 j^{(\epsilon_2)} a_2 j^{(\epsilon_3)} a_3 j^{(\epsilon_4)} \cdots j^{(\epsilon_n)} a_n ) 
=
\sum_{i_1, \dots, i_n \in \{-1, 1\}} (-1)^{k(i)}
\phi'(\lambda_{i_1} a_1 \lambda_{i_2}  a_2 \lambda_{i_3}  a_3 \lambda_{i_4}  \cdots \lambda_{i_n}  a_n )
\end{align*}
where in the sum we require that $i_l = -1$ whenever $\epsilon_l = -1$ and $k(i)$ is the number of times $i_l = -1$ minus the number of times $\epsilon_l = -1$.

Given a tuple $(i_1, i_2, \dots, i_n) \in \{-1, 1\}^{n}$ we create a non-crossing partition $\sigma_i$ by starting a new block at $l$ if $i_l = -1$, as in Notation \ref{nota:epsilon notation}. We next separate  the proof of $($\ref{prop:boolean cumulantsii}$)$ into two parts: the first is when $\epsilon_1 = -1$ and then $\sigma_\epsilon \in I(n)$; the second is when $\epsilon_1 = 1$ and then $\sigma_\epsilon \in \CI(n)$.

To start, let us assume that $\epsilon_1 = -1$. With this assumption, the tuples $(i_1, \dots, i_n)$ parametrize $I(n)$. The condition that $i_l = -1$ whenever $\epsilon_l = -1$ is equivalent to the condition that $\sigma_i \leq \sigma_\epsilon$. So summing over $i_1, \dots, i_n$ is the same as summing over $\sigma \in I(n)$ with $\sigma \leq \sigma_\epsilon$. With this notation we have $(-1)^{k(i)} = (-1)^{\#(\sigma_i) - \#(\sigma_\epsilon)}\ab = \mu_B(\sigma_i, \sigma_\epsilon)$ where $\mu_B$ is the Möbius function of $I(n)$.

Let  $\sigma \in I(n)$ be such that $\sigma \leq \sigma_\epsilon$. For each block $V = \{ l_1, \dots, l_m\}$ of $\sigma$ let $a_V = a_{l_1} \cdots a_{l_m}$. Then we have by $(i)$
\[
\phi'(\lambda_{i_1} a_1 \lambda_{i_2}  a_2 \lambda_{i_3}  a_3 \lambda_{i_4}  \cdots \lambda_{i_n}  a_n) = \phi'(j a_{V_1} j a_{V_2} j \cdots j a_{V_{k}}) = \phi'(j)\phi_{\sigma}(a_1, \dots, a_n). 
\]
Hence 
 \[
\phi'(j^{(\epsilon_1)} a_1 j^{(\epsilon_2)} a_2 j^{(\epsilon_3)} a_3 j^{(\epsilon_4)} \cdots j^{(\epsilon_n
)} a_n) 
=
\phi'(j)
\mathop{\sum_{\sigma \in I(n)}}_{\sigma \leq \sigma_\epsilon} \mu_B(\sigma, \sigma_\epsilon) \phi_\sigma(a_1, \dots, a_n)
=
\phi'(j)\beta_{\sigma_\epsilon}(a_1, \dots, a_n),
 \]
where the last equality follows from \cite[Prop.~10.6]{NS}.

Next, let us assume that $\epsilon_1 = 1$, then $\sigma_\epsilon \in \CI(n)$. The tuples $i_1, \dots, i_n \in \{-1, 1\}$ parametrize $\CI(n)$. The condition that $i_l = -1$ whenever $\epsilon_l = -1$ is equivalent to the condition that $\sigma_i \leq \sigma_\epsilon$. So summing over $i_1, \dots, i_n$ is the same as summing over $\sigma \in \CI(n)$ with $\sigma \leq \sigma_\epsilon$. Moreover, as above, $(-1)^{k(i)} = (-1)^{\#(\sigma_i) - \#(\sigma_\epsilon)}$.
Let  $\sigma \in \CI(n)$ be such that $\sigma \leq \sigma_\epsilon$. For each block $V = \{ l_1, \dots, l_m\}$ of $\sigma$ let $a_V = a_{l_1} \cdots a_{l_m}$. Let us write $V_1 = V_{1,+} \cup V_{1, -}$ where $V_1$ is the block of $\sigma$ containing $1$ and $n$, $V_{1,+}$ contains 1 and the points of $V_1$ moving clockwise, and $V_{1,-}$ contains $n$ and the points of $V_1$ moving counterclockwise.
Then we have by $($\ref{prop:boolean cumulantsi}$)$
\[
\phi'(\lambda_{i_1} a_1 \lambda_{i_2}  a_2 \lambda_{i_3}  a_3 \lambda_{i_4}  \cdots \lambda_{i_n}  a_n) = \phi'(a_{V_1,+} j a_{V_2} j \cdots j a_{V_{1,-}}) = \phi'(j)\phi_{\sigma}(a_1, \dots, a_n). 
\]

Hence 
\begin{align*}
\phi'(j^{(\epsilon_1)} a_1 j^{(\epsilon_2)} a_2 j^{(\epsilon_3)} a_3 j^{(\epsilon_4)} &\cdots j^{(\epsilon_n
)} a_n) 
=
\phi'(j)
\mathop{\sum_{\sigma \in \CI(n)}}_{\sigma \leq \sigma_\epsilon}  
(-1)^{\#(\sigma_i) - \#(\sigma_\epsilon)}   \phi_\sigma(a_1, \dots, a_n)\\
&\mbox{}\stackrel{(*)}{=}
\phi'(j)\beta_{\sigma_\epsilon}(a_1, \dots, a_n).
\end{align*}
To justify the last equality $(*)$,  note that each block of $\sigma_\epsilon$ is an interval (or a cyclic interval), so we can apply \cite[Prop.~10.6]{NS} to each interval and use that $\sigma \mapsto (-1)^{\#(\sigma) - \#(\sigma_\epsilon)}$ is multiplicative across intervals in the sense of \cite[Def.~10.16]{NS}.

To prove $($\ref{prop:boolean cumulantsii.v}$)$, note that the first part of the proof of $($\ref{prop:boolean cumulantsii}$)$ shows that $\sigma_\epsilon \in I(n)$ and here we  have by the same argument that
\[
\phi'(\lambda_{i_1} a_1 \lambda_{i_2}  a_2 \lambda_{i_3}  a_3 \lambda_{i_4}  \cdots \lambda_{i_n}  a_n j) = \phi'(j)\phi_{\sigma}(a_1, \dots, a_n). 
\]
This proves $($\ref{prop:boolean cumulantsii.v}$)$.

Proof of $($\ref{prop:boolean cumulantsiii}$)$. As above, we let $k(i)$ be the number of times $i_l = -1$. When $k(i) = 0$ we have 
$\phi'(\lambda_{i_1} a_1 \cdots \lambda_{i_n} a_n) = \phi'(a_1 \cdots a_n)$ and when $k(i) \geq 1$ we have 
$\phi'(\lambda_{i_1} a_1 \cdots \lambda_{i_n} a_n) = \phi_{\sigma_i}(a_1 \cdots a_n)$. Thus when $k(i) = 1$ there are $n$ tuples $(i_1, \dots, i_n)$ that all produce $\phi(a_1 \cdots a_n)$, because in these cases $\sigma_i = 1_n$. Then
\begin{align*}\lefteqn{
\phi'(j^\perp a_1 \cdots j^\perp a_n)
 =
\sum_{i_1, \dots, i_n \in \{-1, 1\}} (-1)^{k(i)}
\phi'(\lambda_{i_1} a_1 \cdots \lambda_{i_n} a_n)   } \\
& \mbox{} =
\phi'(a_1 \cdots a_n) + \mathop{\sum_{i_1, \dots, i_n \in \{-1, 1\}}}_{k(i) \geq 1} (-1)^{k(i)}
\phi'(\lambda_{i_1} a_1 \cdots \lambda_{i_n} a_n)  \\
& \mbox{} =
\phi'(a_1 \cdots a_n) - n \phi'(j) \phi(a_1 \cdots a_n)
+ \mathop{\sum_{i_1, \dots, i_n \in \{-1, 1\}}}_{k(i) \geq 2} (-1)^{k(i)}
\phi'(\lambda_{i_1} a_1 \cdots \lambda_{i_n} a_n) \\
& \mbox{} =
\phi'(a_1 \cdots a_n) - n \phi'(j)  \phi(a_1 \cdots a_n)
+
\phi'(j) \mathop{\sum_{\sigma \in \CI(n)}}_{\#(\sigma) \geq 2} (-1)^{\#(\sigma)}
\phi_{\sigma}(a_1, \dots,  a_n) \\
& \mbox{} =
\phi'(a_1 \cdots a_n) - (n-1) \phi'(j) \phi(a_1 \cdots a_n)
+
\phi'(j)\mathop{\sum_{\sigma \in \CI(n)}}_{} (-1)^{\#(\sigma)}
\phi_{\sigma}(a_1, \dots,  a_n) \\
&\mbox{} =
\phi'(a_1 \cdots a_n) - \phi'(j) \tilde\beta_n(a_1, \dots, a_n).
\end{align*}
\end{proof}

\begin{rem}\label{rem:spectral shift}
We mention here a surprising relation between an spectral shift functions in operator theory \cite{lifshits},  and \cite[see p. 138]{krein} and Boolean cumulants \cite[Prop. 3.2.1]{agp}, \cite[\S 6.2]{fujie-hasebe2}, and the Markov transform \cite[\S 6]{kerov}, \cite{fujie-hasebe}, and \cite{acg} .

We can form an abstraction of these examples using infinitesimal operators; indeed let $(\cA, \phi, \phi')$ be an infinitesimal probability space and $a, p \in \cA$ infinitesimally free with $p = p^2$ an idempotent with $\phi(p) = 1$ and $\phi'(p) = -1$ (i.e. $p^\perp$ is \emph{infinitesimal}). Let $G_\tau = \sum_{n=0}^\infty \tilde{m}_n z^{-(n+1)}$ be the formal power series with $\tilde{m}_n = \phi'(a^n) - \phi'((pap)^n)$ and $G_\mu(z) = \sum_{n=0}^\infty m_n z^{-(n+1)}$ be the formal power series with $m_n = \phi(a^n)$, then 
\[
G_\tau = -\frac{G_\mu'}{G_\mu}. 
\]

Krein \cite{krein} found an analytic representation of this relation in the form of spectral shift functions. Let $A$ be a self-adjoint bounded operator and $J$ be a rank 1 projection on a Hilbert space $\mathcal{H}$. Let $\xi \in \mathcal{H}$ be such that $\| \xi \| = 1$ and $J\xi = \xi$. Let $\mu$ be the spectral measure of $A$ with respect to the vector state $\omega_\xi(T) = \langle T\xi, \xi \rangle$: for all continuous functions $f$ on $\mathrm{Sp}(A)$ we have $\int f(t) \, d\mu(t) = \omega_\xi(f(A))$. The theorem of Krein, \cite[see p. 138]{krein}, states that there is a finite signed Borel measure $\tau$ such that $G_\tau = -G_\mu'/G_\mu$ and 
\[
\Tr\big(\, f(A) - f(J^\perp A J^\perp)\, \big) 
=
\int f(t) \, d\tau(t). 
\]
In Kerov's notation, $\tau$ is a Raleigh measure and $\mu$ is the Markov transform of $\tau$. This same result is discussed in Kerov \cite[\S 6]{kerov}, where another proof is given. Another example of a Markov transform  is given by Bufetov in  \cite[Thm.~2]{bufetov} where $\mu$ is the Marchenko-Pastur law and $\tau$ is the limiting eigenvalue distribution of a principal minor of a Wishart matrix.

\end{rem}

Let us recall that $\psi(a) = \phi'(aj)/\phi'(j)$. Then, by applying Proposition \ref{prop:boolean cumulants}, we obtain the following result.
\begin{cor}\label{cor:boolean cumulants}
Suppose $(\cA, \phi, \phi')$ is an infinitesimal probability space, $j \in \cA$ is an infinitesimal idempotent with $\phi'(j) \not = 0$, and $\cB \subseteq \cA$ is a unital subalgebra infinitesimally free from $j$. 
Then for all $a_1, \dots, a_n \in \cB$, and $\epsilon_1,\dots,\epsilon_n\in \{-1,1\}$, we have 
\begin{equation}\label{eqn: cor-boolean cumulants}
\psi(j^{(\epsilon_1)} a_1 j^{(\epsilon_2)}a_2 \cdots  j^{(\epsilon_n)}a_n) = \psi(j^{(\epsilon_1)} a_1 j^{(\epsilon_2)}a_2 \cdots  j^{(\epsilon_n)}a_n j) = \beta_{\sigma_\epsilon}(a_1, \dots, a_n).
\end{equation}
\end{cor}
\begin{thm}\label{thm:boolean independence}
Suppose $(\cA_i)_{i\in I}$ are unital algebras of $\cA$ that are infinitesimally free from $j$. Then 
$(j\mathcal{J}(\cA_i)j)_{i\in I}$ are Boolean independent with respect to $\psi.$
If we further assume that $(\cA_i)_{i\in I}$ are free, then we have $(\mathcal{J}_a(\cA_i))_{i\in I}$ are Boolean independent with respect to $\psi$. 
\end{thm}

\begin{proof}
Following Corollary \ref{cor:boolean cumulants}, we immediately obtain that $(j\mathcal{J}(\cA_i)j)_{i\in I}$ are Boolean independent Indeed,  if we let $$x_k=ja_1^{(k)}j^{(\epsilon_2^{(k)})}\cdots j^{(\epsilon_{m_k}^{(k)})}a_{m_k}^{(k)}j 
 \quad\text{ where }\quad \epsilon_2^{(k)},\dots,\epsilon_{m_k}^{(k)}=\pm1,\text{ and }a_1^{(k)},\dots,a_{m_k}^{(k)}\in\cA,
 $$
then by applying Equation \ref{eqn: cor-boolean cumulants}, we can easily to see that
$$\psi(x_1\cdots x_n)=\psi(x_1)\cdots \psi(x_n)
\text{ for each } 1\leq i \leq n.$$ 

Now we assume that $(\cA_i)_{i\in I}$ are free. Let $x_1,\dots,x_n$ be elements in $\mathcal{J}_a(\cA)$ where $x_k\in \mathcal{J}({A}_{i_k})$ with $i_1\neq i_2\neq \cdots \neq i_n$. To prove Boolean independence it is sufficient, by linearity,  to show that
\begin{equation}\label{eqn:Boolean indep}
    \psi(x_1\cdots x_n)=\psi(x_1)\cdots \psi(x_n).
\end{equation}
Following the definition of $\mathcal{J}_a(\cB)$, we note that 
for each $k$, 
$$
x_k=j^{(\varepsilon^{(k)})}a_1^{(k)}j^{(\varepsilon_1^{(k)})}\cdots a_{m_k}^{(k)}j^{(\varepsilon_{m_k}^{(k)})} \text{ for some }a_1^{(k)},\dots,a_{m_k}^{(k)}\in \mathcal{A}_{i_k} \text{ and }
\varepsilon^{(k)}\neq\varepsilon_1^{(k)}\neq \cdots \neq \varepsilon_{m_k}^{(k)}. 
$$
Firstly, we may assume $\varepsilon^{(k+1)}=\varepsilon_{m_k}^{(k)}$ for all $k$; otherwise both sides of Equation \eqref{eq:boolean_independence} are trivially zero. In addition, if all $\varepsilon^{(k)}=\varepsilon_{m_k}^{(k)}= -1$ for all $k$, then 
$(\mathcal{J}_a(\cA)_{i_k})_k$ are Boolean independent by applying Corollary \ref{cor:boolean cumulants} directly. Indeed, 
$$
\psi(x_1\cdots x_n)=\prod_{k=1}^n\prod_{s=1}^{m_k/2}\beta_2(a_{2s-1}^{(k)},a_{2s}^{(k)})=\psi(x_1)\cdots \psi(x_n).
$$
Now, we will show that both sides of Equation \eqref{eqn:Boolean indep} are zero whenever there is $k$ such that $\varepsilon^{(k)}=1$ or $\varepsilon_{m_k}^{(k)}=1$. Without loss of generality, we assume $\varepsilon^{(k)}=1$. Then we note that
\begin{eqnarray*}
    \psi(x_1\cdots x_n)&=&\psi(Z_1 \cdot ja_{m_{k-1}}^{(k-1)}\jp a_{1}^{(k)}j\cdot Z_2) \\
    &=& \beta_{\sigma_1}(a_1^{(1)},\dots,a_{m_{k-1}-1}^{(k-1)})\cdot \beta_2(a_{m_{k-1}}^{(k-1)},a_1^{(k)})\cdot \beta_{\sigma_2}(a_2^{(k)},\dots, a_{m_n}^{(n)})
\end{eqnarray*}
where $Z_1=x_1\cdots x_{k-2}j^{(\varepsilon^{(k-1)})}a_1^{(k-1)}j^{(\varepsilon_1^{(k-1)})}\cdots a_{m_{k-1}-1}^{(k-1)}j$ and
$
Z_2 = ja_2^{(k)}\cdots a_{m_k}^{(k)}j^{(\varepsilon_{m_k}^{(k)})} x_{k+1}\cdots x_n,
$
and $\sigma_i$ is the interval partition respect to $Z_i$ for $i=1,2$. Since $a_{m_{k-1}}^{(k-1)}$ and $a_1^{(k)}$ are free, we have $\beta_2(a_{m_{k-1}}^{(k-1)},a_1^{(k)})=0$, which implies that $\psi(x_1\cdots x_n)=0$. 

On the other hand, 
\begin{eqnarray*}
\psi(x_k)&=&\psi(\jp a_1^{(k)}ja_2^{(k)}\cdots  a_{m_k}^{(k)}j^{(\varepsilon_{m_k}^{(k)})}) \\ 
&=& \psi(a_1^{(k)}ja_2^{(k)}\cdots  a_{m_k}^{(k)}j^{(\varepsilon_{m_k}^{(k)})})-\psi(ja_1^{(k)}ja_2^{(k)}\cdots  a_{m_k}^{(k)}j^{(\varepsilon_{m_k}^{(k)})}) =0
\end{eqnarray*}
due to the fact that $\psi(x)=\psi(jx)$ for all $x\in \mathcal{J}(\mathcal{B})$ whenever $\cB$ is any algebra infinitesimally free from $j$. Hence, we conclude that both sides of Equation \eqref{eqn:Boolean indep} are zero, which complete the proof.  
\end{proof}

\begin{thm}\label{thm:monotone independence}
Suppose that $\cB$ and $\mathcal{C}$ are unital subalgebras of $\cA$ and that $\{\cB,\mathcal{C}\} $ are infinitesimally free from $j$. Then $(j\mathcal{J}(\cB)j, \mathcal{C})$ are monotone independent. Furthermore, if $\cB$ and $\mathcal{C}$ are free, then $(\mathcal{J}_a(\cB),  \mathcal{C})$ are monotone independent. 
\end{thm}
\begin{proof}
That $(j\mathcal{J}(\cB)j)\prec \mathcal{C}$ follows directly from Corollary \ref{cor:boolean cumulants}  and the following argument. Indeed,  if we let $$x_k=ja_1^{(k)}j^{(\epsilon_2^{(k)})}\cdots j^{(\epsilon_{m_k}^{(k)})}a_{m_k}^{(k)}j 
 \quad\text{ where }\quad \epsilon_2^{(k)},\dots,\epsilon_{m_k}^{(k)}=\pm1,\text{ and }a_1^{(k)},\dots,a_{m_k}^{(k)}\in\cA,
 $$
then by applying Equation \ref{eqn: cor-boolean cumulants}, we can easily to see that
$$\psi(x_1\cdots x_n)=\psi(x_1\cdots x_{i-1}\psi(x_i)x_{i+1}\cdots x_n)
\text{ for each } 1\leq i \leq n.$$ 

Now assume $\cB$ and $\mathcal{C}$ are free, we are going to show $(\mathcal{J}_a(\cB),\mathcal{C}$) are monotone independent.

We shall show that for any $k$ we have
\begin{equation}\label{eqn:mono-indep}
    \psi(c_0x_1c_1\cdots x_nc_n)=\psi(c_0x_1c_1\cdots x_{k}\psi(c_k)x_{k+1}\cdots c_{n-1} x_n c_n).
\end{equation}
where $c_0,\dots,c_n\in\mathcal{C}$ and $x_1,\dots,x_n\in\mathcal{J}_a(\cB)$. Again following the definition of $\mathcal{J}_a(\cB)$, we have 
$$
x_k=j^{(\varepsilon^{(k)})}a_1^{(k)}j^{(\varepsilon_1^{(k)})}\cdots a_{m_k}^{(k)}j^{(\varepsilon_{m_k}^{(k)})} \text{ for some }a_1^{(k)},\dots,a_{m_k}^{(k)}\in \cB \text{ and }
\varepsilon^{(k)}\neq\varepsilon_1^{(k)}\neq \cdots \neq \varepsilon_{m_k}^{(k)}. 
$$
By linearity and induction on $n$ this will prove monotone independence.

Without lost of generality, we assume $\varepsilon^{(k+1)}=\varepsilon_{m_k}^{(k)}$; otherwise, both sides of Equation \eqref{eqn:mono-indep} are trivially zero. For the case $\varepsilon^{(k+1)}=\varepsilon_{m_k}^{(k)}=-1$, 
\begin{eqnarray*}
    &&\psi(c_0x_1c_1\cdots x_k c_k x_{k+1}\cdots x_nc_n) \\
    &=& \beta_{\sigma_1}(c_0,a_1^{(1)},\dots,a_{m_{k}}^{(\varepsilon_{m_k}^{(k)})})\beta_1(c_k)\beta_{\sigma_2}(a_1^{(k+1)},\dots,c_n) \\
    &=& \beta_{\sigma_1}(c_0,a_1^{(1)},\dots,a_{m_{k}}^{(\varepsilon_{m_k}^{(k)})})\psi(c_k)\beta_{\sigma_2}(a_1^{(k+1)},\dots,c_n) \\
    &=& \psi(c_k)\psi(c_0x_1c_1\cdots x_k x_{k+1}\cdots x_nc_n)
\end{eqnarray*} 
where 
$\sigma_1$ is the interval partition with respect to $c_0x_1c_1\cdots x_k$ and $\sigma_2$ is the interval partition with respect to $x_{k+1}c_{k+1}\cdots x_n c_n$. 
Now, let us consider the case $\varepsilon^{(k+1)}=\varepsilon_{m_k}^{(k)}=1$, we have by Corollary \ref{cor:boolean cumulants} 
\begin{eqnarray*}
    &&\psi(c_0x_1c_1\cdots x_k c_k x_{k+1}\cdots x_nc_n) \\
    &=& \psi(c_0x_1c_1\cdots a_{m_{k}-1}^{(\varepsilon_{m_k-1}^{(k)})}ja_{m_{k}}^{(\varepsilon_{m_k}^{(k)})}\jp c_k \jp a_1^{(k+1)}ja_2^{(k+1)}\cdots x_nc_n  ) \\
    &=&\beta_{\sigma_1}(c_0,x_1,\dots,a_{m_{k}-1}^{(\varepsilon_{m_k-1}^{(k)})})\beta_3(a_{m_{k}}^{(\varepsilon_{m_k}^{(k)})},c_k, a_1^{(k+1)})\beta_{\sigma_2}(a_2^{(k+1)},\dots,c_n)
\end{eqnarray*}
where $\sigma_1$ is the interval partition with respect to $c_0x_1c_1\cdots a_{m_{k}-1}^{(\varepsilon_{m_k-1}^{(k)})}j$ and $\sigma_2$ is the interval partition with respect to $ja_2^{(k+1)}\cdots x_n c_n$. The freeness of $\cB$ and $\mathcal{C}$ implies that 
$$
\beta_3(a_{m_{k}}^{(\varepsilon_{m_k}^{(k)})},c_k ,a_1^{(k+1)}) = \beta_2(a_{m_{k}}^{(\varepsilon_{m_k}^{(k)})},a_1^{(k+1)})\beta_1(c_k) = \beta_2(a_{m_{k}}^{(\varepsilon_{m_k}^{(k)})},a_1^{(k+1)})\psi(c_k).
$$
Thus,
\begin{eqnarray*}
    &&\psi(c_0x_1c_1\cdots x_k c_kx_{k+1}\cdots x_nc_n) \\
    &=& \psi(c_k)\beta_{\sigma_1}(c_0,x_1,\dots,a_{m_{k}-1}^{(\varepsilon_{m_k-1}^{(k)})})\beta_2(a_{m_{k}}^{(\varepsilon_{m_k}^{(k)})},a_1^{(k+1)})\beta_{\sigma_2}(a_2^{(k+1)},\dots,c_n) \\
    &=& \psi(c_0x_1c_1\cdots a_{m_{k}-1}^{(\varepsilon_{m_k-1}^{(k)})}ja_{m_{k}}^{(\varepsilon_{m_k}^{(k)})}\jp 1 \jp a_1^{(k+1)}ja_2^{(k+1)}\cdots x_nc_n ) \\
    &=& \psi(c_k)\psi(c_0x_1c_1\cdots x_k x_{k+1}\cdots x_nc_n).
\end{eqnarray*}
Hence, we show Equation \eqref{eqn:mono-indep} holds. 
\end{proof}


\section{The Distributions of Some Polynomials of Boolean Independent Operators}
         \label{sec:boolean-independent}

Suppose that $(\cA,\phi)$ is a non-commutative probability space and $x_1$ and $x_2$ are elements in $\cA$ that are Boolean independent. In this final section, we show how to find the distribution of $ax_1x_2+bx_2x_1$ where $a$ and $b$ are complex numbers. In particular, we will obtain the distributions of the commutator and anticommutator of Boolean independent random variables.

First note that by applying \cite[Prop.~10.11]{NS} and the analogous argument of \cite[Thm.~11.12]{NS}, we have the products as arguments formula in the Boolean, we state it as follows: 

\begin{thm}\label{thm:boolean cumulants products}
Suppose $(\cA,\varphi)$ is a non-commutative probability space. Let $m,n\in \mathbb{N}$ and $1\leq i_1 <i_2 <\cdots < i_m=n$ be given, and we set
$$
\hat{0}_m=\{(1,\dots,i_1),\dots,(i_{m-1}+1,\dots,i_m)\}\in I(m).
$$
Then we have
\begin{equation*}
\beta_m(a_1\cdots a_{i_1},\dots,a_{i_{m-1}+1}\cdots a_{i_m})=\sum_{
\substack{ \pi\in I(n)\\  \pi\vee \hat{0}_m=1_n}}\beta_{\pi}(a_1,\dots,a_n).    
\end{equation*}
\end{thm}
\subsection{The linear span Case}

Let $(\cA,\varphi)$ be a non-commutative probability space. Assume $x_1$ and $x_2$ are elements in $\cA$ that are Boolean independent and $a,b\in\mathbb{C}$. We let $g=ax_1x_2+bx_2x_1$. In this Subsection, we will provide the distribution of $g$ by computing its Boolean cumulants and also moments. 

Let us define $c=x_1x_2$, and we set $c^{(1)}=x_1x_2$ and $c^{(-1)}=x_2x_1.$ We first show the following result: 
\begin{lem}\label{Lem: Boolean poly cumulants}
For all $\epsilon_1,\dots,\epsilon_n\in \{-1,1\}$, we have
\begin{eqnarray*}
\beta_n(c^{(\epsilon_1)},\dots,c^{(\epsilon_n)}) =0
\end{eqnarray*}
unless $\epsilon_i+\epsilon_{i+1}=0$ for $1\leq i <n$. Moreover, 
\begin{eqnarray*}
\beta_n(c^{(\epsilon_1)},\dots,c^{(\epsilon_n)})=
\begin{cases}
\beta_2(x_2)\beta_1(x_1)^2\Big(\beta_2(x_1)\beta_2(x_2)\Big)^{(n-2)/2}
 & \text{ if }(\epsilon_1,\dots,\epsilon_n)=(1,-1,\dots,1,-1) \\
\beta_2(x_1)\beta_1(x_2)^2\Big(\beta_2(x_1)\beta_2(x_2)\Big)^{(n-2)/2}
 & \text{ if }(\epsilon_1,\dots,\epsilon_n)=(-1,1,\dots,-1,1) \\
\end{cases}
\end{eqnarray*}
for $n$ is even, and 
\begin{equation*}
\beta_n(c^{(\epsilon_1)},\dots,c^{(\epsilon_n)}) =
\beta_1(x_1)\beta_1(x_2)\Big(\beta_2(x_1)\beta_2(x_2)\Big)^{(n-1)/2}
\end{equation*}
for $n$ is odd with $(\epsilon_1,\dots,\epsilon_n)=(1,-1,\dots,1,-1,1)$ or $(-1,1,\dots,-1,1,-1)$.
\end{lem}
\begin{proof}Let $\rho$ be the partition $\{(1, 2), (3, 4), \dots, (2 n - 1, 2n )\}$:

\begin{center}
\begin{tikzpicture}[anchor=base, baseline]
\node[above] at (-0.75,0.50) {$\rho = \mbox{}$};
\node[above] at (0.00, 0.75) {$1$};
\node[above] at (0.50, 0.75) {$2$};
\node[above] at (1.00, 0.75) {$3$};
\node[above] at (1.50, 0.75) {$4$};
\node[above] at (2.00, 0.75) {$5$};
\node[above] at (2.50, 0.75) {$6$};
\node[above] at (3.00, 0.75) {$7$};
\node[above] at (3.50, 0.75) {$8$};
\node[above] at (4.25, 0.75) {$\cdots$};
\node[above] at (5.40, 0.75) {$2n-1$};
\node[above] at (6.50, 0.75) {$2n$};
\draw  (0.00 , 0.75) -- (0.00, 0.50) -- (0.50, 0.50) -- (0.50, 0.75);
\draw  (1.00 , 0.75) -- (1.00, 0.50) -- (1.50, 0.50) -- (1.50, 0.75);
\draw  (2.00 , 0.75) -- (2.00, 0.50) -- (2.50, 0.50) -- (2.50, 0.75);
\draw  (3.00 , 0.75) -- (3.00, 0.50) -- (3.50, 0.50) -- (3.50, 0.75);
\draw  (5.50 , 0.75) -- (5.50, 0.50) -- (6.50, 0.50) -- (6.50, 0.75);
\end{tikzpicture}
\end{center}
Next, let $x_{-1} = x_2$. Then $c^{(\epsilon)} = x_{\epsilon} x_{-\epsilon}$. So if $\epsilon = 1$ we have $c^{(\epsilon)} = x_1 x_2$ and if $\epsilon = -1$ we have $c^{(\epsilon)} = x_2 x_1$. Then with this notation, Theorem \ref{thm:boolean cumulants products} becomes
\begin{equation}\label{eq:boolean cumulant products}
\beta_n(c^{(\epsilon_1)},\dots,c^{(\epsilon_n)}) =
\mathop{\sum_{\pi \in I(2n)}}_{\pi \vee \rho = 1_{2n}}
\beta_\pi (x_{\epsilon_1}, x_{-\epsilon_1}, x_{\epsilon_2}, x_{-\epsilon_2}, \dots, x_{\epsilon_n}, x_{-\epsilon_n}).
\end{equation}
In order to have 
\[
\beta_\pi (x_{\epsilon_1}, x_{-\epsilon_1}, x_{\epsilon_2}, x_{-\epsilon_2}, \dots, x_{\epsilon_n}, x_{-\epsilon_n}) \neq  0
\]
we must have, by the vanishing of mixed cumulants, that $|V \cap W| \leq 1$, for each block $V$ of $\pi$ and each block $W$ of $\rho$. On the other hand, in order have $\pi \vee \rho = 1_{2n}$, we must have that $2k + 1 \sim_\pi 2k + 2$, for each $k$, since $\pi$ is an interval partition. These two conditions together imply that for each $k$, $x_{-\epsilon_k}$ and $x_{\epsilon_{k+1}}$ are the same operator, which in turn means that $-\epsilon_k = \epsilon_{k+1}$, as claimed. 

Now suppose that $\epsilon_i+\epsilon_{i+1}=0$ for $1\leq i <n$, let us compute $\beta_n(c^{(\epsilon_1)},\dots,c^{(\epsilon_n)})$. From the discussion above, 
equation (\ref{eq:boolean cumulant products}) means that there is only one $\pi$ which contributes to (\ref{eq:boolean cumulant products}); $\pi = \{(1),(2, 3), \dots, (2n-2, 2n -1), (2n)\}$:
\begin{center}
\begin{tikzpicture}[anchor=base, baseline]
\node[above] at (-0.75,0.50) {$\pi = \mbox{}$};
\node[above] at (0.00, 0.75) {$1$};
\node[above] at (0.50, 0.75) {$2$};
\node[above] at (1.00, 0.75) {$3$};
\node[above] at (1.50, 0.75) {$4$};
\node[above] at (2.00, 0.75) {$5$};
\node[above] at (2.50, 0.75) {$6$};
\node[above] at (3.00, 0.75) {$7$};
\node[above] at (4.00, 0.75) {$\cdots$};
\node[above] at (5.40, 0.75) {$2n-2$};
\node[above] at (7.00, 0.75) {$2n-1$};
\node[above] at (8.00, 0.75) {$2n$};
\draw  (0.00, 0.75) -- (0.00, 0.50); 
\draw  (0.50, 0.75) -- (0.50, 0.50) -- (1.00, 0.50) -- (1.00, 0.75);
\draw  (1.50, 0.75) -- (1.50, 0.50) -- (2.00, 0.50) -- (2.00, 0.75);
\draw  (2.50, 0.75) -- (2.50, 0.50) -- (3.00, 0.50) -- (3.00, 0.75);
\draw  (5.50, 0.75) -- (5.50, 0.50) -- (7.1, 0.50) -- (7.10, 0.75);
\draw  (8.00, 0.75) -- (8.00, 0.50);
\end{tikzpicture}
\end{center} 
When $n$ is odd, we must have $n -1$ blocks of size 2, half with $x_1$ and half with $x_2$. This contributes $(\beta_2(x_1) \beta_2(x_2))^{(n -1 )/2}$. Of the two singletons, one is a $x_1$ and the other must be a $x_2$. Thus the total contribution is $\beta_1(x_1) \beta_1(x_2)(\beta_2(x_1) \beta_2(x_2))^{(n -1 )/2}$. 

When $n$ is even, we also have $n-1$ blocks of size $2$. If $\epsilon_1 = 1$, then we have  $n/2$ blocks with $x_2$  and the remaining $(n-2)/2$ are with $x_1$. The two singletons are both with $x_1$. Thus the total contribution is $\beta_1(x_1)^2 \beta_2(x_1)^{(n-2)/2}\beta_2(x_2)^{n/2} $.

When $n$ is even and $\epsilon_1 = -1$, it is the other way around. There are $n/2$ blocks of size $2$ with $x_1$ and $(n-2)/2$ blocks of size $2$ with $x_1$, contributing $\beta_2(x_1)^{n/2} \beta_2(x_2)^{(n-2)/2}$. Finally, there are two blocks of size $1$, one both with $x_2$. This proves the claim.
\end{proof}

\begin{rem}
If $(\cA,\varphi)$ is a $*$-probability space and if we further assume $x_1$ and $x_2$ are self-adjoint, then by Lemma  \ref{Lem: Boolean poly cumulants}, we have $c$ is $\eta$-diagonal (for more details on $\eta$-diagonal, see \cite{BNNS, MH}).
\end{rem}
\begin{thm}\label{thm: Boolean poly cumulants}
Suppose $x_1$ and $x_2$ are Boolean independent, and $g=ax_1x_2+bx_2x_1$ for $a,b\in\mathbb{C}$, then 
$$
\beta_n(g)= \begin{cases}   ab(\beta_2(x_1)\beta_1(x_2)^2+\beta_2(x_2)\beta_1(x_1)^2)\Big(ab\beta_2(x_1)\beta_2(x_2)\Big)^{(n-2)/2} & \text{ if }n  \text{ is even} \\    (a+b)\beta_1(x_1)\beta_1(x_2)\Big(ab\beta_2(x_1)\beta_2(x_2)\Big)^{(n-1)/2} & \text{ if }n \text{ is odd}.    
\end{cases}
$$
\end{thm}
\begin{proof}
Suppose $n$ is even, we note that 
\begin{eqnarray*}
\beta_n(g)&=&\beta_n(ac^{(1)}+bc^{(-1)}) \\
&=&\beta_n(ac^{(1)},bc^{(-1)},\cdots,ac^{(1)},bc^{(-1)}) + \beta_n(bc^{(-1)},ac^{(1)},\dots,bc^{(-1)},ac^{(1)})+\text{ other terms} \\
&=& (ab)^{\frac{n}{2}}[\beta_n(c^{(1)},c^{(-1)},\dots,c^{(1)},c^{(-1)})+\beta_n(c^{(-1)},c^{(1)},\dots,c^{(-1)},c^{(1)})]+\text{other terms}.
\end{eqnarray*}
By Lemma \ref{Lem: Boolean poly cumulants}, we further obtain
\begin{eqnarray*}
\beta_n(g)&=&(ab)^{\frac{n}{2}}[\beta_2(x_2)\beta_1(x_1)^2\Big(\beta_2(x_1)\beta_2(x_2)\Big)^{\frac{n-2}{2}}+\beta_2(x_1)\beta_1(x_2)^2\Big(\beta_2(x_1)\beta_2(x_2)\Big)^{\frac{n-2}{2}}] + 0 \\
&=& ab(\beta_2(x_1)\beta_1(x_2)^2+\beta_2(x_2)\beta_1(x_1)^2)\Big(ab\beta_2(x_1)\beta_2(x_2)\Big)^{\frac{n-2}{2}}.
\end{eqnarray*}
The proof for the case of odd $n$ is similar.    
\end{proof}
Note that for the anti-commutator $p=x_1x_2+x_2x_1$ and the commutator $q=i(x_1x_2-x_2x_1)$, we can easily to see that
$$
\beta_1(p)= 2\beta_1(x_1)\beta_1(x_2), \text{ and }\beta_2(p)= \beta_2(q)=\beta_1(x_1)^2\beta_2(x_2)+\beta_1(x_2)^2\beta_2(x_1). 
$$
Then following Theorem \ref{thm: Boolean poly cumulants}, we have the distributions of $p$ and $q$ immediately by letting $(a,b)=(1,1)$ and $(a,b)=(i,-i)$ respectively. We state it as below.  
\begin{cor}
With $p$ and $q$ as above
$$
\beta_n(p) =
\begin{cases}
    \beta_2(p) (\beta_2(x_1)\beta_2(x_2))^{(n-2)/2} & \text{if } n \text{ is even}\\
     \beta_1(p) (\beta_2(x_1)\beta_2(x_2))^{(n-1)/2}  & \text{if } n \text{ is odd}
\end{cases}
$$
and 
\[\beta_n(q) =
\begin{cases}
    \beta_2(q) (\beta_2(x_1)\beta_2(x_2))^{(n-2)/2} & \text{if } n \text{ is even}\\
    0                                               & \text{if } n \text{ is odd}
\end{cases}.
\]
\end{cor}
We will provide the formula for the moments of $g=ax_1x_2+bx_2x_1$ in terms of cumulants via Theorem \ref{thm: Boolean poly cumulants}. Observe that following the moment-cumulant formula, $m_n(g)=\sum_{\pi \in I(n)}\beta_{\pi}(g)$ where 
$$
\beta_{2k}(g)=d^{k-1}\alpha_2, \text{ and }\beta_{2k+1}(g)=d^{k}\alpha_1
$$
with $d=ab\beta_2(x_1)\beta_2(x_2)$, $\alpha_1=\frac{1}{2}(a+b)\beta_1(p)$ and $\alpha_2=ab\beta_2(p)$. To simplify our notation, now we will write $m_n$ for $m_n(g)$ and $\beta_n$ for $\beta_n(g)$. 

\begin{nota}
Let $\pi\in I(n)$, where the first block of $\pi$ is denoted as $V_1$. Then we define $I_e(n)$ (and $I_o(n)$ respectively) as the set of interval partitions such that $|V_1|$ is even (odd respectively). We further define:
$$\gamma_{o,n}=\sum_{\pi\in I_o(n)}\beta_{\pi}\quad\text{ and }\quad \gamma_{e,n}=\sum_{\pi\in I_e(n)}\beta_{\pi}.
$$ 
\end{nota}

We first show that $m_k,\gamma_{e,k},$ and $\gamma_{o,k}$ satisfy the following relation.
\begin{lem}\label{lem: m-gamma-gamma-relation}
For each $k$, we have
\begin{eqnarray*}
m_k=\gamma_{o,k}+\gamma_{e,k}, \quad \gamma_{e,k+1}=\frac{\alpha_2}{\alpha_1}\gamma_{o,k} \quad\text{ and }\quad\gamma_{o,k+1}=\alpha_1 m_k+d\frac{\alpha_1}{\alpha_2}\gamma_{e,k}.    
\end{eqnarray*}
\end{lem}
\begin{proof}
It is clear that
\begin{equation}\label{eqn: m_k - gamma - formula}
m_k=\sum_{\pi\in I(k)}\beta_{\pi}= \sum_{\pi\in I_o(k)}\beta_{\pi} + \sum_{\pi\in I_e(k)}\beta_{\pi} =\gamma_{o,k}+\gamma_{e,k}.
\end{equation}
In addition, we note that each $\pi\in I_e(k+1)$ is generated from a $\sigma\in I_o(k)$ by adding one element into the first block of $\sigma$. To be precise, if we write  
$$
\sigma=\{(1,\dots,i_1),(i_1+1,\dots,i_2),\dots,(i_{k-1}+1,\dots,i_k)\} \text{ with }i_1=2s+1\text{ for some $s$, and }i_k=k.
$$
Then $\pi=\{(1,\dots,i_1+1),(i_1+2,\dots,i_2+1),\dots,(i_{k-1}+2,\dots,i_k+1)\}. 
$ Moreover, we observe $$\beta_{i_1+1}=\beta_{2s+2}=d^s\alpha_2=\frac{\alpha_2}{\alpha_1}d^s\alpha_1\beta_{2s+1}=\frac{\alpha_2}{\alpha_1}\beta_{i_1}.$$ 
Hence, we have
$$
\beta_{\pi}= \beta_{i_1+1}\beta_{i_2-i_1}\cdots \beta_{i_k-i_{k-1}} = \frac{\alpha_2}{\alpha_1}\beta_{i_1}\beta_{i_2-i_1}\cdots \beta_{i_k-i_{k-1}}=\frac{\alpha_2}{\alpha_1}\beta_{\sigma}.
$$
Thus, we conclude $\gamma_{e,k+1}=\frac{\alpha_2}{\alpha_1}\gamma_{o,k}$. 

On the other hand, note that each $\pi\in I_{o}(k+1)$ is generated either from a $\sigma_1\in I_{e}(k)$ by adding one element into the first block of $\sigma_1$ or adding singleton as the first block for an interval partition $\sigma_2\in I(k)$. Following the analogous argument, one gets
$$
\gamma_{o,k+1}=\alpha_1 m_k+d\frac{\alpha_1}{\alpha_2}\gamma_{e,k}.
$$
\end{proof}

\begin{thm}\label{thm: moment of g}
With $g=ax_1x_2+bx_2x_1$, we have 
\begin{eqnarray*}
m_k(g)=\frac{1}{\omega}\big(\frac{\alpha_1+\omega}{2}\big)^{k-1}\big(\alpha_1\frac{\alpha_1+\omega}{2}+\alpha_2\big)-\frac{1}{\omega}\big(\frac{\alpha_1-\omega}{2}\big)^{k-1}\big(\alpha_1\frac{\alpha_1-\omega}{2}+\alpha_2\big)
\end{eqnarray*} 
where $\omega=\sqrt{\alpha_1^2+4(d+\alpha_2)}$ with $d=ab\beta_2(x_1)\beta_2(x_2)$, $\alpha_1=\frac{1}{2}(a+b)\beta_1(p)$ and $\alpha_2=ab\beta_2(p)$. In other words, $$
g\sim \frac{\alpha_1\theta_1+\alpha_2}{\omega\theta_1}\delta_{\theta_1}-\frac{\alpha_1\theta_2+\alpha_2}{\omega\theta_2}\delta_{\theta_2}
\text{ where }\theta_1=\frac{\alpha_1+\omega}{2}, \text{ and }\theta_2=\frac{\alpha_1-\omega}{2}.
$$
\end{thm}

\begin{proof}
Firstly, it is clear that $\gamma_{o,1}=\alpha_1$ and $\gamma_{o,2}=\alpha_1^2$.
Then by Lemma \ref{lem: m-gamma-gamma-relation}, we can deduce that 
$$
\frac{\gamma_{o,k+2}}{\alpha_1}-\frac{d}{\alpha_1}\gamma_{o,k}=\gamma_{o,k+1}+\frac{\alpha_2}{\alpha_1}\gamma_{o,k}, 
$$
which implies 
$$
\gamma_{o,k+2}-\alpha_1 \gamma_{o,k+1}-(d+\alpha_2)\gamma_{o,k}=0.
$$
By solving the characteristic polynomial $x^2-\alpha_1x-(d+\alpha_2)=0$, we obtain the general solution
\begin{equation*}
    \gamma_{o,k}=C_1 \big(\frac{\alpha_1+\omega}{2}\big)^k+C_2\big(\frac{\alpha_1-\omega}{2}\big)^k 
\end{equation*}
for some constants $C_1$ and $C_2$, where $\omega=\sqrt{\alpha_1^2+4(d+\alpha_2)}.$ Then by plugging in the initial conditions $\gamma_{o,1}=\alpha_1$ and $\gamma_{o,2}=\alpha_1^2$, one gets 
$$
\gamma_{o,k}=\frac{\alpha_1}{\omega} \big(\frac{\alpha_1+\omega}{2}\big)^k-\frac{\alpha_1}{\omega}\big(\frac{\alpha_1-\omega}{2}\big)^k. 
$$
Therefore, 
$$
\gamma_{e,k}=\frac{\alpha_2}{\alpha_1}\gamma_{o,k-1}=\frac{\alpha_2}{\omega} \big(\frac{\alpha_1+\omega}{2}\big)^{k-1}-\frac{\alpha_2}{\omega}\big(\frac{\alpha_1-\omega}{2}\big)^{k-1};
$$
hence, we finally get 
$$
m_k=\gamma_{o,k}+\gamma_{e,k}=\frac{1}{\omega}\big(\frac{\alpha_1+\omega}{2}\big)^{k-1}\big(\alpha_1\frac{\alpha_1+\omega}{2}+\alpha_2\big)-\frac{1}{\omega}\big(\frac{\alpha_1-\omega}{2}\big)^{k-1}\big(\alpha_1\frac{\alpha_1-\omega}{2}+\alpha_2\big). 
$$
\end{proof}

\subsection{The Infinitesimal cases}
In this subsection, we will give the infinitesimal distribution of $g=ax_1x_2+bx_2x_1$ whenever $x_1$ and $x_2$ are infinitesimally Boolean independent in an infinitesimal probability space $(\cA,\phi,\phi')$. 

Let $(\cA,\phi,\phi')$ be an infinitesimal probability space, and suppose elements $x_1$ and $x_2$ in $\cA$ are infinitesimally Boolean independent. By Theorem \ref{thm:Upper-Boolean In}, $X_1=\begin{bmatrix}
    x_1 & 0 \\ 0 & x_1
\end{bmatrix}$ and $X_2=\begin{bmatrix}
    x_2 & 0 \\ 0 & x_2
\end{bmatrix}$ are Boolean independent in $(\widetilde{\cA},\widetilde{\mathbb{C}},\widetilde{\phi})$. Let $G=aX_1X_2+bX_2X_1$. We 
note that $\widetilde{\mathbb{C}}$ is a commutative algebra; as a result, we have that $\beta_n^{\widetilde{\phi}}(G)$ follows the same computation as we proved for $\beta_n(g)$. To be precise, $\beta_n^{\widetilde{\phi}}(G)$ is now $2 \times 2$ upper triangular matrices, and we have 
\begin{equation}\label{Eqn: Upper-Boolean p}
\beta_n^{\widetilde{\phi}}(G) =
\begin{cases}
    ab(\beta_2^{\widetilde{\phi}}(X_1)\beta_1^{\widetilde{\phi}}(X_2)^2+ \beta_2^{\widetilde{\phi}}(X_2)\beta_1^{\widetilde{\phi}}(X_1)^2)\Big(ab\beta_2^{\widetilde{\phi}}(X_1)\beta_2^{\widetilde{\phi}}(X_2)\Big)^{(n-2)/2} & \text{if } n \text{ is even}\\
     (a+b)\beta_1^{\widetilde{\phi}}(X_1)\beta_1^{\widetilde{\phi}}(X_2) \Big(ab\beta_2^{\widetilde{\phi}}(X_1)\beta_2^{\widetilde{\phi}}(X_2)\Big)^{(n-1)/2}  & \text{if } n \text{ is odd}
\end{cases}
\end{equation}
Suppose $p=x_1x_2+x_2x_1$, then we recall that $\beta_1(p)=2\beta_1(x_1)\beta_1(x_2)$ and $\beta_2(p)=\beta_1(x_1)^2\beta_2(x_2)+\beta_1(x_2)^2\beta_2(x_1)$. It can easily to deduce that
\begin{eqnarray*}
    \beta_1'(p)&=& 2\beta_1'(x_1)\beta_1(x_2)+2\beta_1(x_1)\beta_1'(x_2), \\
    \beta_2'(p)&=& 2\beta_1(x_1)\beta_1'(x_1)\beta_2(x_2)+\beta_1(x_1)^2\beta_2'(x_2)+2\beta_1(x_2)\beta_1'(x_2)\beta_2(x_1)+\beta_1(x_2)^2\beta_2'(x_1). 
\end{eqnarray*}
\begin{thm}\label{thm: inf-Boolean cumulant}
Suppose that $x_1$ and $x_2$ are infinitesimally Boolean independent, and let $g=ax_1x_2+bx_2x_1$, we have  
$$\beta_n'(g) =
\begin{cases}
     (ab)^{\frac{n}{2}}(\beta_2(x_1)\beta_2(x_2))^{\frac{n-4}{2}}\Big[\frac{n-2}{2}\beta_2(p)\big(\beta_2(x_1)\beta_2'(x_2)+\beta_2'(x_1)\beta_2(x_2)\big)\Big]& \\
    \qquad\qquad\qquad\qquad\mbox{}+(ab)^{\frac{n}{2}} 
     \beta_2'(p)\, \big[\beta_2(x_1)\beta_2(x_2)\big]^{\frac{n-2}{2}}& \text{if } n \text{ is even}\\
     (a+b)(ab)^{\frac{n-1}{2}}(\beta_2(x_1)\beta_2(x_2))^{\frac{n-3}{2}}\Big[\frac{n-1}{2}\beta_1(p)\big(\beta_2(x_1)\beta_2'(x_2)+\beta_2'(x_1)\beta_2(x_2)\big)\Big]& \\
    \qquad\qquad\qquad\qquad\mbox{}+  (a+b)(ab)^{\frac{n-1}{2}}\beta'_1(p)\,\big[\beta_2(x_1)\beta_2(x_2)\big]^{\frac{n-1}{2}}  & \text{if } n \text{ is odd}
\end{cases}.
$$
\end{thm}
\begin{proof}
Observe that the $(1,2)$-entry of $\beta_n^{\widetilde{\phi}}(G)$ is $\beta_n'(g)$. On the other hand, 
\begin{eqnarray*}
     &&ab(\beta_2^{\widetilde{\phi}}(X_1)\beta_1^{\widetilde{\phi}}(X_2)^2+ \beta_2^{\widetilde{\phi}}(X_2)\beta_1^{\widetilde{\phi}}(X_1)^2)\Big(ab\beta_2^{\widetilde{\phi}}(X_1)\beta_2^{\widetilde{\phi}}(X_2)\Big)^{(n-2)/2}\\ 
     &=&ab\begin{bmatrix}
         \beta_2(p) & \beta_2'(p) \\
         0 & \beta_2(p)
     \end{bmatrix} \Big(ab\begin{bmatrix}
         \beta_2(x_1) & \beta_2'(x_1) \\ 
         0 & \beta_2(x_1)
     \end{bmatrix}\cdot \begin{bmatrix}
         \beta_2(x_2) & \beta_2'(x_2) \\ 
         0 & \beta_2(x_2)
     \end{bmatrix}\Big)^{(n-2)/2} \\
     &=& ab\begin{bmatrix}
         \beta_2(p) & \beta_2'(p) \\
         0 & \beta_2(p)
     \end{bmatrix}\Big(ab\begin{bmatrix}
         \beta_2(x_1)\beta_2(x_2) & \beta_2(x_1)\beta_2'(x_2)+\beta_2'(x_1)\beta_2(x_2) \\ 0 & \beta_2(x_1)\beta_2(x_2) 
     \end{bmatrix}\Big)^{(n-2)/2}  \\
     &=& (ab)^{\frac{n}{2}}\begin{bmatrix}
         \beta_2(p) & \beta_2'(p) \\
         0 & \beta_2(p)
     \end{bmatrix} \\
     && \mbox{} \times \begin{bmatrix}
         (\beta_2(x_1)\beta_2(x_2))^{(n-2)/2} & \frac{n-2}{2}(\beta_2(x_1)\beta_2(x_2))^{(n-4)/2}\big(\beta_2(x_1)\beta_2'(x_2)+\beta_2'(x_1)\beta_2(x_2)\big) \\ 0 &  (\beta_2(x_1)\beta_2(x_2))^{(n-2)/2}\end{bmatrix} 
\end{eqnarray*}
Thus, the $(1,2)$-entry of $ab(\beta_2^{\widetilde{\phi}}(X_1)\beta_1^{\widetilde{\phi}}(X_2)^2+ \beta_2^{\widetilde{\phi}}(X_2)\beta_1^{\widetilde{\phi}}(X_1)^2)\Big(ab\beta_2^{\widetilde{\phi}}(X_1)\beta_2^{\widetilde{\phi}}(X_2)\Big)^{(n-2)/2}$ is
\[
 (ab)^{\frac{n}{2}}\frac{n-2}{2}\beta_2(p)   (\beta_2(x_1)\beta_2(x_2))^{\frac{n-4}{2}}\Big(\beta_2(x_1)\beta_2'(x_2)+ \beta_2'(x_1)\beta_2(x_2)\Big)+(ab)^{\frac{n}{2}}\beta_2'(p) (\beta_2(x_1)\beta_2(x_2))^{\frac{n-2}{2}}  
\]
\[\mbox{}=(ab)^{\frac{n}{2}}
(\beta_2(x_1)\beta_2(x_2))^{\frac{n-4}{2}}\bigg(\frac{n-2}{2}\beta_2(p)\Big(\beta_2(x_1)\beta_2'(x_2)+\beta_2'(x_1)\beta_2(x_2)\Big)\bigg)
\]
\[
\mbox{} +
(ab)^{\frac{n}{2}}\beta_2'(p) \big(\beta_2(x_1)\beta_2(x_2)\big)^{\frac{n-2}{2}}. 
\]
Applying Equation \eqref{Eqn: Upper-Boolean p}, we obtain
\begin{align*}
\beta_n'(g) & =  (ab)^{\frac{n}{2}}
(\beta_2(x_1)\beta_2(x_2))^{\frac{n-4}{2}}\bigg(\frac{n-2}{2}\beta_2(p)\Big(\beta_2(x_1)\beta_2'(x_2)+\beta_2'(x_1)\beta_2(x_2)\Big)\bigg) \\
&\quad +(ab)^{\frac{n}{2}}\beta_2'(p) \big(\beta_2(x_1)\beta_2(x_2)\big)^{\frac{n-2}{2}}. 
\end{align*}
whenever $n$ is even. By applying Equations \eqref{Eqn: Upper-Boolean p}, the remaining cases can be done via an analogous argument. 
\end{proof}

Following Theorem \ref{thm: inf-Boolean cumulant}, we have the infinitesimal distributions of the anti-commutator $p=x_1x_2+x_2x_1$ and the commutator $q=i(x_1x_2-x_2x_1)$ as follows.  
\begin{cor}
$$\beta_n'(p) =
\begin{cases}
     (\beta_2(x_1)\beta_2(x_2))^{\frac{n-4}{2}}\Big[\frac{n-2}{2}\beta_2(p)\big(\beta_2(x_1)\beta_2'(x_2)+\beta_2'(x_1)\beta_2(x_2)\big)\Big]
     & \\
    \qquad\qquad\qquad\qquad\qquad\qquad\mbox{}+ 
     \beta_2'(p)\, \big[\beta_2(x_1)\beta_2(x_2)\big]^{\frac{n-2}{2}} & \text{if } n \text{ is even}\\
     (\beta_2(x_1)\beta_2(x_2))^{\frac{n-3}{2}}\Big[\frac{n-1}{2}\beta_1(p)\big(\beta_2(x_1)\beta_2'(x_2)+\beta_2'(x_1)\beta_2(x_2)\big)\Big]& \\
    \qquad\qquad\qquad\qquad\qquad\qquad\mbox{}+ \beta'_1(p)\,\big[\beta_2(x_1)\beta_2(x_2)\big]^{\frac{n-1}{2}}  & \text{if } n \text{ is odd}
\end{cases}
$$
and
$$
\beta_n'(q) =
\begin{cases}
    (\beta_2(x_1)\beta_2(x_2))^{\frac{n-4}{2}}\Big[\frac{n-2}{2}\beta_2(q)\big(\beta_2(x_1)\beta_2'(x_2)
    + \beta_2'(x_1)\beta_2(x_2)\big)\Big] & \\
    \qquad\qquad\qquad\qquad\qquad\qquad\mbox{}+ \beta_2'(q)\,\big[\beta_2(x_1)\beta_2(x_2)\big]^{\frac{n-2}{2}} & \text{if } n \text{ is even}\\
     0  & \text{if } n \text{ is odd}
\end{cases}.
$$
\end{cor}

\section*{Acknowledgement}

The authors would like to express our sincere gratitude to the anonymous reviewers for their insightful comments and constructive suggestions, which greatly contributed to the improvement of this manuscript. The revision of this paper has also benefited from the comments and suggestions of Yen Jen Cheng, Nicolas Gilliers, and Takahiro Hasebe.


\thebottomline\end{document}